\documentclass[a4paper,12pt]{amsart}
\usepackage{amssymb}
\usepackage{amsmath, amsthm, amscd, amsfonts, amssymb, graphicx, color}
\usepackage{amsmath, amsthm, amscd, amsfonts, amssymb, graphicx, color}

\usepackage[cp1250]{inputenc}
\usepackage[T1]{fontenc}

\newtheorem{theorem}{Theorem}[section]
\newtheorem{lemma}[theorem]{Lemma}
\newtheorem{proposition}[theorem]{Proposition}
\newtheorem{corollary}[theorem]{Corollary}
\theoremstyle{definition}
\newtheorem{definition}[theorem]{Definition}

\theoremstyle{remark}
\newtheorem{remark}[theorem]{Remark}
\numberwithin{equation}{section}
\title { IzbaĂ„Ĺ¤eno }

\begin{document}
\author{Stefan Ivkovi\'{c} }
 
\vspace{15pt}
	
\title{Spectral Fredholm theory and transitivity in Banach bimodules}

\maketitle
\begin{abstract}
	
	In this paper, we extend Fredholm theory in von Neumann algebras established by Breuer to spectral Fredholm theory. We consider 2 by 2 upper triangular operator matrices with coefficients in a von Neumann algebra and give the relationship between the generalized essential spectra in the sense of Breuer of such matrices and of their diagonal entries. Next, we prove that if a generalized Fredholm operator in the sense of Breuer has $0$ as an isolated point of its spectrum, then the corresponding spectral projection is finite. Finally, we define generalized B-Fredholm operator in a von Neumann algebra as a generalization in the sense of Breuer of B-Fredholm operators on Hilbert and Banach spaces defined by Berkani. We provide sufficient conditions under which a sum of a generalized B-Fredholm operator and a finite operator in a von Neumann algebra is again a generalized B-Fredholm operator. Finally, motivated by the connections between supercyclicity and semi-Fredholm theory, in the last section of the paper we characterize disjoint supercyclic and disjoint Furstenberg semi-transitive operators on a large class of Banach bimodules. 
\end{abstract}

\vspace{15pt}

\begin{flushleft}
	\end{flushleft} 

\vspace{15pt}

\begin{flushleft}
\end{flushleft}

\vspace{30pt}

\section{Introduction}

The Fredholm and semi-Fredholm theory on Hilbert and Banach spaces started by studying the integral equations introduced in the pioneering work by Fredholm in 1903 in \cite{F}. After that, the abstract theory of Fredholm and semi-Fredholm operators on Hilbert and Banach spaces was further developed in numerous papers and books such as \cite{AP1}, \cite{AP2}. 
In addition to classical semi-Fredholm theory on Hilbert and Banach spaces, several generalizations of this theory have been considered. Breuer for example started the development of Fredholm theory in von Neumann algebras as a generalization of the classical Fredholm theory for operators on Hilbert spaces. In \cite{BR} and \cite{BR2} he introduced the notion of a Fredholm operator in a von Neumann algebra and established its main properties. In \cite{BJMA} we went further in this direction and extended Breuer`s Fredholm theory to semi-Fredholm and semi-Weyl theory in von Neumann algebras. The interest for considering these generalizations comes from the theory of pseudo differential operators acting on manifolds. The classical theory can be applied in the case of compact manifolds, but not in the case of non-compact ones. Even operators on Euclidian spaces are hard to study, for example Laplacian is not Fredholm.  Orthogonal projections onto kernels and cokernels of many bounded linear operators on Hilbert spaces are not finite rank projections in the classical sense, but they are still finite projections in an appropriate von Neumann algebra. Therefore, many operators that are not semi-Fredholm in the classical sense may become semi-Fredholm in a more general sense if we consider them as elements of an appropriate von- Neumann algebra. Hence, by studying these generalized semi-Fredholm operators, we get a proper extension of the classical semi-Fredholm theory to new classes of operators. This is in fact the main reason for considering these generalizations. \\

The main aim of this paper is to extend Fredholm theory in von Neumann algebras established in \cite{BR} and \cite{BR2} to spectral Fredholm theory in von Neumann algebras generalizing in this setting the results from the classical spectral semi-Fredholm theory for operators on Hilbert and Banach spaces. In the first result in Section 4 we consider 2 by 2 upper triangular operator matrices with coefficients in a von Neumann algebra and describe the relationship between the essential spectra of such matrices and of their diagonal entries. These essential spectra which we consider are induced by the class of generalized Fredholm operators in the sense of Breuer. Next, in Section 4 we consider isolated points of the spectrum of an operator $F$ in a von Neumann algebra $ \mathcal{A} .$ We prove that if $F$ is generalized Fredholm operator in the sense of Breuer and has $0$ as an isolated point of its spectrum, then the spectral projection corresponding to $0$ is a finite operator in $ \mathcal{A} .$ Then we introduce a concept of generalized Browder operators in $ \mathcal{A} $ as a proper generalization of the classical Browder operators on Hilbert spaces (Fredholm operators with finite ascent and descent), and we show that the class of these generalized Browder operators is a subclass of generalized Fredholm operators in the sense of Breuer. As a consequence of our result regarding finiteness of spectral projections corresponding to isolated points of the spectrum, we prove that if a generalized Fredholm operator in the sense of Breuer has $0$ as an isolated point of its spectrum, then it is generalized Browder in the sense of our definition. This is a generalization of the well known result from the classical Fredholm theory on Hilbert spaces. Finally, at the end of Section 4 we consider compressions of operators in a von Neumann algebra induced by finite projections and obtain a generalization in this setting of the results by Zemanek given in \cite{ZE} concerning the relationship between the spectra of an operator and of its compressions.\\
In Section 5 we introduce the notion of generalized B-Fredholm operators in a von Neumann algebra as a generalization in the sense of Breuer of B-Fredholm operators on Hilbert and Banach spaces defined by Berkani in \cite{BS} and \cite{BM}. We prove that a sum of a generalized B-Fredholm operator $T$ and a finite operator $F$ in a von Neumann algebra $ \mathcal{A} $ is a again a generalized B-Fredholm operator in the sense of our definition provided that there exists some $m$ such that  $ Im (T+F)^{n}  $ is closed for all $n \geq m.$ 

Now, there is a close connection between supercyclicity and semi-Fredholm theory, see \cite{cao, Fredholm-3}. On the other hand, von-Neumann algebras are just a special case of Banach bimodules. All these facts motivated us to study and characterize disjoinr supercyclic and disjoint Furstenberg semi-transitive operators on a large class of Banach bimodules, extending in this way the results from \cite{AOFA, arxiv} to some new classes of Banach bimodules.   

\section{Preliminaries}

In this section we shall recall some definitions, concepts and results from Fredholm theory in  $C^{*}-$algebras which will be needed in the rest of the paper. \\
We start with the following definitions.  

\begin{definition} \label{d 09}
	\cite[Definition 1.1]{KL} Let $\mathcal{A} $ be an unital $C^{*}-$algebra and  $\mathcal{F} \subseteq \mathcal{A} $ be a subalgebra which satisfies the following conditions:\\
	(i) $\mathcal{F} $ is a selfadjoint ideal in $\mathcal{A} ,$ i.e. for all $a \in \mathcal{A}, b \in \mathcal{F}  $ there holds $ab,ba \in \mathcal{F} ,$ and $a \in \mathcal{F} $ implies $a^{*} \in \mathcal{F} ;$\\
	(ii) There is an approximate unit $p_{\alpha}$ in the norm topology for $\mathcal{F}$ consisting of projections;\\
	(iii) If $p,q \in \mathcal{F} $ are projections, then there exists $v \in \mathcal{A} ,$ such that $vv^{*}=q $ and $v^{*}v \perp p, $ i.e. $v^{*}v + p $ is a projections as well;\\
	Such a family we shall call \textit{finite type elements}. In further, we shall denote it by $\mathcal{F} .$
\end{definition}

\begin{definition} \label{d 10} \cite[Definition 1.2]{KL}
	Let $\mathcal{A} $ be an unital $C^{*}-$algebra, and let $\mathcal{F} \subseteq \mathcal{A}$ be an algebra of finite type elements.
\end{definition}

In the set $ \text{ Proj}(\mathcal{F}) $ we define the equivalence relation:  
$$p \sim q \Leftrightarrow \exists v \in \mathcal{A} \text{ } vv^{*}=p,\text{ } v^{*}v=p, $$
i.e. Murray - von Neumann equivalence. The set $ S(\mathcal{F})=\text{Proj}(\mathcal{F})\text{ }/\sim $ is a commutative semigroup with respect to addition, and the set,$K(\mathcal{F})=G(S(\mathcal{F})),$ where $G$ denotes the Grothendic functor, is a commutative group.

\begin{definition} \label{d 11} \cite[Definition 2.1]{KL}
	Let $a \in \mathcal{A} $ and let $ p,q \in \mathcal{F}$  be projections. We say that $a$ is invertible up to pair $(p, q)$ if the element $a^{\prime}=(1-q)a(1-p) $ is invertible, i.e., if there is some $b \in \mathcal{A} $ with $b=(1-p)b(1-q) $ (and immediately $bp=0, $ $pb=0,$ $b=(1-p)b=b(1-q )$) such that 
	$$a^{\prime}b=1-q, \text{  } ba^{\prime}=1-p. $$
	We refer to such $b$ as almost inverse of $a,$ or $(p,q)-$inverse of $a.$
\end{definition}

\begin{definition} \label{d 12} \cite[Definition 2.2]{KL}
	Let $\mathcal{F}$ be finite type elements. We say that $a \in \mathcal{A} $ is of Fredholm type (or abstract Fredholm element) if there are $p,q \in \mathcal{F} $ such that $a$ is invertible up to $(p,q).$ The index of the element $a$ (or abstract index) is the element of the group $K(\mathcal{F}) $ defined by
	$$\text{ind}(a)=([p],[q]) \in K(\mathcal{F}), $$
	or less formally
	$$\text{ind}(a)=[p]-[q]. $$
\end{definition} 

Next we recall the following lemmas. 

\begin{lemma} \label{l 01}\cite[Lemma 1]{BJMA}
	Let $a \in \mathcal{A} $ and $p,q \in \mathcal{F} .$ Then $a$ is invertible up to pair $(p,q) $ if and only if $a^{*}$ is invertible up to pair $(q,p) .$ 
\end{lemma}

\begin{lemma} \label{l 02}\cite[Lemma 2]{BJMA}
	Let $a \in \mathcal{A} $ and $p,q,p^{\prime},q^{\prime} $ be projections in $\mathcal{A} .$ Suppose that $p,q,p^{\prime} \in \mathcal{F} .$ If $a$ is invertible up to pair $(p,q) $ and also invertible up to pair $(p^{\prime},q^{\prime}) ,$ then $q^{\prime} \in \mathcal{F} .$ 
% Recenica 2 */* ====================================================================
	Similarly, if instead of $p,q,p^{\prime} $ we have that $p,q,q^{\prime} \in \mathcal{F} ,$ then we must have that $p^{\prime} $ as well.
%====================================================================================
\end{lemma}

%==================================================================================
From the proof of Lemma \ref{l 02} we can also deduce the following corollaries. 
\begin{corollary} \label{cor 3.3}\cite[Corollary 4]{BJMA}
	Let $ a \in \mathcal{A} .$ If $a$ is invertible both up to pair $(p,q)$ and up to pair $(p , q^{\prime}),$ then $q \sim q^{\prime} .$
\end{corollary}

\begin{corollary} \label{r8 cor2.4}
	Let $a \in \mathcal{A} .$ If $a$ is invertible up to pairs $(p, q)$ and $(p^{\prime}, q^{\prime}) $ where $p, p^{\prime} \in \mathcal{F} ,$ then there exist projections $\tilde{q} , \tilde{q}^{\prime} , \tilde{\tilde{q}} $ and $\tilde{\tilde{q}}^{\prime} $ in $\mathcal{A} $ such that $\tilde{\tilde{q}} ,\tilde{\tilde{q}}^{\prime} \in \mathcal{F} ,$ 
	$ \tilde{q} \tilde{\tilde{q}}= \tilde{q}^{\prime} \tilde{\tilde{q}}=0, \text{  } q \sim  \tilde{q}, q^{\prime}  \sim  \tilde{q}^{\prime}, $
	$  \tilde{q}+ \tilde{\tilde{q}} \sim \tilde{q}^{\prime} + \tilde{\tilde{q}}^{\prime}    $
	and $a$ is invertible up to pairs $(p, \tilde{q}) $ and $ (p^{\prime \prime}, \tilde{q}^{\prime})$ for some projection $ p^{\prime \prime} \sim p^{\prime}.$ Similar statement hold if we instead of $p$ and $p^{\prime} $ have that $ q, q^{\prime} \in \mathcal{F},$ however, in this case there exist projections 
	$\tilde{p}, \tilde{p}^{\prime}, \tilde{ \tilde{p} }, \tilde{ \tilde{p} }^{\prime} $ in $\mathcal{A} $ such that $\tilde{ \tilde{p} }, \tilde{ \tilde{p} }^{\prime} \in \mathcal{F}, $ 
	$\tilde{p} \tilde{ \tilde{p} }=\tilde{p}^{\prime} \tilde{ \tilde{p} }^{\prime}=0,$ 
	$p \sim \tilde{p}, p^{\prime} \sim \tilde{p}^{\prime},\tilde{p}+\tilde{ \tilde{p} } \sim \tilde{p}^{\prime}+ \tilde{ \tilde{p} }^{\prime}  $ and $a$ is invertible up to pairs $( \tilde{p},q) $ and $(\tilde{p}^{\prime},q^{\prime \prime}) $ for some projection $q^{\prime \prime} \sim q^{\prime} .$
\end{corollary}

\begin{proof}
	By \cite[Proposition 2.8]{KL} there exists projection $\tilde{q} $ in $\mathcal{A} $ such that $\tilde{q} \sim q ,$ $\tilde{q}a(1-p)=0$ and $a$ is invertible up to pair $ (p,\tilde{q} ).$ Then, by the proof of Lemma \ref{l 02}, there is an approximate unit $ \lbrace p_{\alpha} \rbrace$ for $ \mathcal{F} ,$ projections $ p^{\prime \prime}, q^{\prime \prime} $ in $\mathcal{A} $ and nets of projections $  \lbrace q_{\alpha} \rbrace$ and $  \lbrace q_{\alpha}^{\prime \prime} \rbrace$ in $\mathcal{A} $ such that $p^{\prime} \sim p^{\prime \prime} \leq p_{\alpha}$ for all $ \alpha,$ $q^{\prime \prime} \sim q^{\prime} ,$ $a$ is invertible up to pair $(p^{\prime \prime}, q^{\prime \prime}) $ and $q_{\alpha} - \tilde{q} \sim p_{\alpha} - p , q_{\alpha}^{\prime \prime} - q^{\prime \prime} \sim p_{\alpha} - p^{\prime \prime}$ $ $ and $q_{\alpha}^{\prime \prime} \sim q_{\alpha} .$ For any fixed $\alpha ,$ set $\tilde{ \tilde{q} } = q_{\alpha} - \tilde{q} ,$ $\tilde{q}^{\prime}= q^{\prime \prime} ,$ $\tilde{ \tilde{q} }^{\prime}=q_{\alpha}^{\prime \prime} - q^{\prime \prime} .$ This proves the first statement. \\
	The second statement can be proved by passing to the adjoints and applying Lemma \ref{l 01}.
\end{proof}

%=====================================================

Now we define semi-Fredholm and semi-Weyl type elements in unital $C^{*}-$algebras.

\begin{definition} \label{d 04}\cite[Definition 5]{BJMA}
	Let $a \in \mathcal{A} .$ We say that $a$ is an upper semi-Fredholm type element if $a$ is invertible up to pair $(p, q)$ where $p \in \mathcal{F} .$ Similarly, we say that $a$ is a lower semi-Fredholm type element, however in this case we assume that $q \in \mathcal{F} $ (and not $p$).
\end{definition}

\begin{definition} \label{r10d 1.1}\cite[Definition 6]{BJMA}
	Let $p, q$ be projections in $\mathcal{A}.$ We will denote $p \preceq q $ if there exists some projection $p^{\prime} $ such that $p^{\prime} \leq q $ and $ p \sim p^{\prime} .$ 
\end{definition}

\begin{definition} \label{r10d 1.3}\cite[Definition 7]{BJMA}
	Let $a \in \mathcal{A} .$ We say that $a$ is an upper semi-Weyl type element if there exist projections $(p, q)$ in $\mathcal{A} $ such that $p \in \mathcal{F} ,$ $p \preceq q $ and $a$ is invertible up to pair $(p, q).$ Similarly we say that $a$ is a lower semi-Weyl type element, only in this case we assume that $q \in \mathcal{F} $ and $q \preceq p .$ Finally, we say that $a$ is a Weyl type element if $a$ is invertible up to pair $(p, q)$ where $p, q \in \mathcal{F} $ and $p \sim q .$ 
\end{definition}
Set\\
$\mathcal{K}\Phi_{+} (\mathcal{A})= \lbrace a \in \mathcal{A} \mid  a \text{ is upper semi-Fredholm type element }  \rbrace ,$\\
$\mathcal{K}\Phi_{-} (\mathcal{A})= \lbrace a \in \mathcal{A} \mid  a \text{  is lower semi-Fredholm type element }  \rbrace ,$\\
$\mathcal{K}\Phi (\mathcal{A})= \lbrace a \in \mathcal{A} \mid  a \text{ is  Fredholm type element }  \rbrace ,$\\
$\mathcal{K}\Phi_{+}^{-} (\mathcal{A})= \lbrace a \in \mathcal{A} \mid  a \text{  is upper semi-Weyl type element }  \rbrace ,$\\
$\mathcal{K}\Phi_{-}^{+} (\mathcal{A})= \lbrace a \in \mathcal{A} \mid a \text{  is lower semi-Weyl type element }  \rbrace ,$\\
$\mathcal{K}\Phi_{0} (\mathcal{A})= \lbrace a \in \mathcal{A} \mid  a \text{  is Weyl type element }  \rbrace .$\\
Notice that by definition we have $\mathcal{K}\Phi_{+}^{-} (\mathcal{A}) \subseteq \mathcal{K}\Phi_{+} (\mathcal{A}), \text{ } \mathcal{K}\Phi_{-}^{+} (\mathcal{A}) \subseteq \mathcal{K}\Phi_{-} (\mathcal{A}) $ and $\mathcal{K}\Phi_{0} (\mathcal{A}) \subseteq \mathcal{K}\Phi (\mathcal{A}) .$\\

From now on and in the rest of this paper, $\mathcal{A}$ denotes a properly infinite von Neumann algebra acting on a Hilbert space $H .$ For a closed subspace $N$ of $H ,$ we let $ P_{N} $ denote the orthogonal projection onto $N.$ By the symbol $\tilde{ \oplus} $ \index{$\tilde{ \oplus} $} we denote the direct sum of closed subspaces of $H$ as given in \cite{MT}.Thus, if $M$ is a closed subspace of $H$ and $M_{1}, M_{2}$ are two closed subspaces of $M,$ we write $M=M_{1} \tilde \oplus M_{2}$ if $M_{1} \cap M_{2}=\lbrace 0 \rbrace$ and $M_{1}+ M_{2}=M.$ If, in addition $M_{1}$  and $M_{2}$ are mutually orthogonal, then we write $M=M_{1} \oplus M_{2}.$
 Also, we let $Proj_{0}(\mathcal{A})$ denote the set of all finite projections in $\mathcal{A},$ (i.e. those projections that are not Murray von Neumann equivalent to any of its subprojections).\\
We recall the notion of $\mathcal{A}$-Fredholm operator, originally introduced by Breuer in \cite{BR}, \cite{BR2}. 

\begin{definition} \cite[Definition 3.1]{KL}
	The operator $T \in \mathcal{A}$ is said to be $\mathcal{A}$-Fredholm if the following holds.\\
	$(i)$  $P_{\ker T}  \in Proj_{0}(\mathcal{A}), $ where $P_{\ker T}$ is the projection onto the subspace $\ker T$.\\
	$(ii)$ There is a projection $E \in Proj_{0}(\mathcal{A}) $ such that $Im(I - E) \subseteq Im T . $\\
	The second condition ensures that $P_{\ker T^{*}}$
	also belongs to $Proj_{0}(\mathcal{A}).$\\
	
	The index of an $\mathcal{A}$-Fredholm operator $T$ is defined as 
	$$index T = dim(\ker T ) - dim(\ker T^{*}) \in I(\mathcal{A}) . $$
	
	Here, $I(\mathcal{A})$ is the so called\textit{ index group} of a von Neumann algebra $\mathcal{A}$ defined as the Grothendieck group of the commutative monoid of all representations of the commutant $\mathcal{A}^{\prime}$ generated by representations of the form $\mathcal{A}^{\prime}  \ni S \mapsto ES =\pi_{E}(S) $ for
	some $E \in Proj_{0}(\mathcal{A}) .$ For a subspace $L,$ its dimension $dim L$ is defined as the $class [\pi_{P_{L}}] \in I(\mathcal{A}) $ of the representation $ \pi_{ P_{L}},$ where $P_{L}$ is the projection onto $L.$
\end{definition}

Next we recall the following characterization of $\mathcal{A}$-Fredholm operators.

\begin{lemma}  \label{r12 l11}\cite[Lemma 22]{BJMA}
	Let $\mathcal{A} $ be a properly infinite von Neumann algebra. Then an operator $T \in \mathcal{A} $ is $\mathcal{A}-$Fredholm in the sense of Breuer if and only if there exist projections $ P,Q \in Proj_{0} (\mathcal{A})$ such that $T$ is invertible up to $(P,Q) .$ 
\end{lemma}

	Let $\mathcal{F} = \mathfrak{m} $ where $m$ is the norm closure of the set of all $S \in \mathcal{A} $ for which $P_{\overline{ImS}} \in Proj_{0} (\mathcal{A}) .$ Then 
	\begin{equation} \label{f1}
		\lbrace P \in \mathfrak{m} \text{ } \vert \text{ } P \text{   is projection  } \rbrace = Proj_{0}(\mathcal{A}). 
	\end{equation}

This relation has been used in the proof of Lemma \ref{r12 l11}.\\

We recall also the following definition. 

\begin{definition} \label{r12 d15}\cite[Definition 9]{BJMA}
	Let $\mathcal{A} $ be a properly infinite von Neumann algebra and $T \in \mathcal{A} .$ We say that $T$ is upper semi$-\mathcal{A}-$Fredholm if there exist projections $P, Q$ in $\mathcal{A}$ such that $T$ is invertible up to $(P,Q)$ where $ P \in Proj_{0} (\mathcal{A}).$ If in addition $P \preceq Q ,$ we say that $T$ upper semi$-\mathcal{A}-$Weyl. Similarly we say that $T$ is lower semi$- \mathcal{A} -$Fredholm an lower semi$-\mathcal{A} -$Weyl, however in this case we assume that $Q \in  Proj_{0} (\mathcal{A})$ and $Q \preceq P .$
\end{definition}

Thanks to the relation (\ref{f1}) we obtain the following useful characterization of semi$- \mathcal{A} -$Fredholm and semi$- \mathcal{A} -$Weyl operators. 

\begin{corollary} \label{r12 c17}\cite[Corollary 23]{BJMA}
	Let $T \in \mathcal{A}.$ Then $T$ is upper (respectively lower) semi-Fredholm type element in $ \mathcal{A}$ with respect to $\mathfrak{m}$  if and only if $T$ is upper (respectively lower) semi$- \mathcal{A} -$Fredholm. Similarly, $T$ is upper (respectively lower) semi-Weyl type element in $\mathcal{A} $ with respect to $\mathfrak{m}$  if and only if $T$ is upper (respectively lower) semi$- \mathcal{A} -$Weyl. Finally, $T$ is Weyl type element in $ \mathcal{A} $ with respect to $\mathfrak{m}$  if and only if $T$ is $\mathcal{A} -$Weyl.
\end{corollary}

\section{On some properties of 2 by 2 operator matrices in von Neumann algebras}

Let $M_{2} (\mathcal{A}) $ denote the von Neumann algebra consisting of 2 by 2  matrices with coefficients in $\mathcal{A} .$ If $\mathcal{A} $ is a properly infinite von Neumannn algebra, then $M_{2} (\mathcal{A}) $ is also a properly infinite von Neumannn algebra.\\
 In this section we shall derive some technical properties of 2 by 2 matrices in a von Neumann algebra which we will use in the rest of the paper.  We start with the following auxiliary technical lemma.

\begin{lemma} \label{r12 l12}
	Let $ \mathcal{A} $ be a properly infinite von Neumann algebra and $P,Q \in Proj_{0} (\mathcal{A}).$ Then 
	$ \left( 
	\begin{pmatrix}
	P & 0 \\
	0 & 0 
	\end{pmatrix} ,
	\begin{pmatrix}
	0 & 0 \\
	0 & Q 
	\end{pmatrix}
	\right)$ 
	$  \in Proj_{0} (M_{2}(\mathcal{A})),$ if and only if $ P,Q \in Proj_{0} (\mathcal{A}).$
\end{lemma}

\begin{proof}
	Let $P \in Proj_{0} (\mathcal{A})$ and assume that there exists a subprojection $ \tilde{P}^{\prime}$ of 
	$
	\begin{pmatrix}
	P & 0 \\
	0 & 0 
	\end{pmatrix}
	$ 
	such that $\tilde{P}^{\prime} \sim$ 
	$
	\begin{pmatrix}
	P & 0 \\
	0 & 0 
	\end{pmatrix}
	.$ Since $\tilde{P} $ is a subprojection of 
	$
	\begin{pmatrix}
	P & 0 \\
	0 & 0 
	\end{pmatrix}
	,$ then $\tilde{P}^{\prime} =$ 
	$
	\begin{pmatrix}
	P^{\prime} & 0 \\
	0 & 0 
	\end{pmatrix}
	$ for some $P^{\prime} \leq P.$ Let $ V \in M_{2} (\mathcal{A})$ such that $ V V^{*}=$
	$
	\begin{pmatrix}
	P & 0 \\
	0 & 0 
	\end{pmatrix}
	$ and $V^{*}V=$
	$
	\begin{pmatrix}
	P^{\prime} & 0 \\
	0 & 0 
	\end{pmatrix}
	.$ 
	Then, if we put $\tilde{P}=$
	$
	\begin{pmatrix}
	P & 0 \\
	0 & 0 
	\end{pmatrix}
	,$ we get that $\tilde{P} V \tilde{P}^{\prime} V^{*} \tilde{P}  =  \tilde{P} V V^{*} V V^{*} \tilde{P}  =  \tilde{P} $ and 
	$\tilde{P}^{\prime} V^{*} \tilde{P} V \tilde{P}^{\prime} = \tilde{P}^{\prime} V^{*} V V^{*} V \tilde{P}^{\prime} = \tilde{P}^{\prime} .$ Moreover, if we write $V$ as 
	$
	\begin{pmatrix}
	V_{1} & V_{2} \\
	V_{3} & V_{4} 
	\end{pmatrix}
	$ 
	for some $V_{1},V_{2},V_{3},V_{4} \in \mathcal{A} ,$ we get that $\tilde{P} V \tilde{P}^{\prime} =$
	$
	\begin{pmatrix}
	PV_{1}P^{\prime} & 0 \\
	0 & 0 
	\end{pmatrix}
	.$  Hence $P= (P V_{1} P^{\prime} ) (P V_{1} P^{\prime} )^{*} $ and $P^{\prime}= (P V_{1} P^{\prime} )^{*} (P V_{1} P^{\prime} ) ,$ so $P \sim P^{\prime} $ which is a contradiction. Thus, we must have that 
	$
	\begin{pmatrix}
	P & 0 \\
	0 & 0 
	\end{pmatrix}
	$ 
	$\in Proj_{0} (M_{2}(\mathcal{A})).$ Conversely, it is obvious that if $P \in Proj (\mathcal{A})$ and $ P \sim P^{\prime}$ for some $P^{\prime} \leq P ,$ then 
	$
	\begin{pmatrix}
	P & 0 \\
	0 & 0 
	\end{pmatrix}
	$
	$ \sim $
	$
	\begin{pmatrix}
	P^{\prime} & 0 \\
	0 & 0 
	\end{pmatrix}
	.$ 
	Hence, if 
	$
	\begin{pmatrix}
	P & 0 \\
	0 & 0 
	\end{pmatrix}
	$ 
	$ \in Proj_{0}(M_{2}(\mathcal{A})) ,$ then $P \in Proj_{0} (\mathcal{A}).$\\
	Similarly we treat the case with 
	$
	\begin{pmatrix}
	0 & 0 \\
	0 & Q 
	\end{pmatrix}
	.$ 
\end{proof}	
Next we recall also the following properties of finite operators in von Neumann algebras. 

\begin{lemma} \label{r12 l13}\cite{BR},\cite{BR2}
	Let $\mathcal{A} $ be a properly infinite von Neumannn algebra and $T \in \mathcal{A} .$ Then $ P_{\overline{Im T}} \in Proj_{0} (\mathcal{A}) $ if and only if $P_{\overline{Im T^{*}}} \in Proj_{0}  (\mathcal{A})$ and in this case $ P_{\overline{Im S_{1} T S_{2}}} \in Proj_{0} (\mathcal{A})$ for all $ S_{1}, S_{2} \in  \mathcal{A}.$
\end{lemma}

\begin{proof}
	If $T \in \mathcal{A} ,$ then $ P_{\overline{Im T}} \sim  P_{{\ker T^{\perp}}} .$ Since $P_{\ker T^{\perp}} = P_{\overline{Im T^{*}}}  ,$ the first statement follows. Now, if $S_{2} \in \mathcal{A} ,$ then $P_{\overline{Im T S_{2}}} \leq P_{\overline{Im T }} ,$ hence we must have $P_{\overline{Im T S_{2}}} \in Proj_{0} (\mathcal{A}) $ if $P_{\overline{Im T}} \in Proj_{0} (\mathcal{A}) .$ By the first statement we also get that $ P_{\overline{ImS_{2}^{*} T^{*}}} \in Proj_{0} (\mathcal{A}) .$ Hence, if in addition $S_{1} \in \mathcal{A} ,$ by repeating the same argument we get that $P_{\overline{Im S_{2}^{*} T^{*} S_{1}^{*}}} \in Proj_{0} (\mathcal{A}) ,$ so $P_{\overline{Im S_{1} T S_{2}}} \in Proj_{0} (\mathcal{A}) .$
\end{proof}
From Lemma \ref{r12 l13} we deduce the following useful corollary.  

\begin{corollary} \label{r12 c14}
	Let $T=$ 
	$
	\begin{pmatrix}
	T_{1} & T_{2} \\
	T_{3} & T_{4} 
	\end{pmatrix}
	$ be an element of $M_{2} (\mathcal{A}) $ where $ \mathcal{A} $ is a properly infinite von Neumann algebra. 
	If $ P_{\overline{Im T}} \in  Proj_{0} (M_{2}(\mathcal{A})) ,$ then 
	$P_{\overline{Im T_{1}}}, P_{\overline{Im T_{2}}} \in Proj_{0} (\mathcal{A}) .$
\end{corollary}

\begin{proof}
	By Lemma  \ref{r12 l13}, if $ \tilde{T}_{1}=$ 
	$
	\begin{pmatrix}
	T_{1} & 0 \\
	0 & 0 
	\end{pmatrix}
	,$ then $P_{\overline{Im \tilde{T}_{1}}}  \in Proj_{0} (M_{2}(\mathcal{A}))$ if $P_{\overline{Im T}} \in Proj_{0} (M_{2}(\mathcal{A})) .$ 
	This is because 
	$ \tilde{T}_{1}=
	\begin{pmatrix}
	1 & 0 \\
	0 & 0 
	\end{pmatrix}
	T
	\begin{pmatrix}
	1 & 0 \\
	0 & 0 
	\end{pmatrix}
	.$
	However, $P_{\overline{Im \tilde{T}_{1}}} =$
	$
	\begin{pmatrix}
	P_{\overline{Im T_{1}}} & 0 \\
	0 & 0 
	\end{pmatrix}
	,$ hence by Lemma \ref{r12 l12}, $ P_{\overline{Im T_{1}}} \in Proj_{0} (\mathcal{A}).$ Similarly we can prove that $P_{\overline{Im T_{4}}} \in Proj_{0} (\mathcal{A}) .$ 
\end{proof}

For an operator $ T^{\prime} \in M_{2}(\mathcal{A})$ we shall simply say that $T^{\prime} $ is $\mathcal{A} -$Fredholm if $T^{\prime} $ is $M_{2} (\mathcal{A}) -$Fredholm. We have the following corollary.
\begin{corollary} \label{r12 c18}
	Let $\mathcal{A} $ be a properly infinite von Neumann algebra and $ T,S \in \mathcal{A}.$ If $T$ is $ \mathcal{A}-$Fredholm, then 
	$
	\begin{pmatrix}
	T & 0 \\
	0 & 1 
	\end{pmatrix} 
	$
	is $ \mathcal{A} -$Fredholm. Similarly, if $S$ is $\mathcal{A} -$Fredholm, then 
	$ 
	\begin{pmatrix}
	1 & 0 \\
	0 & S 
	\end{pmatrix} 
	$ is $ \mathcal{A} -$Fredholm.
\end{corollary}

\begin{proof}
	If $T$ is invertible up to $(P, Q)$ for some projections $P$ and $Q,$ then 
	$
	\begin{pmatrix}
	T & 0 \\
	0 & 1 
	\end{pmatrix} 
	$ is invertible up to 
	$ \left(
	\begin{pmatrix}
	P & 0 \\
	0 & 0 
	\end{pmatrix} 
	,
	\begin{pmatrix}
	Q & 0 \\
	0 & 0 
	\end{pmatrix} 
	\right) .$
	Similarly, if $S$ is invertible up to $(P^{\prime}, Q^{\prime}) ,$ for some projections $P^{\prime}$ and $Q^{\prime},$ then 
	$
	\begin{pmatrix}
	1 & 0 \\
	0 & S 
	\end{pmatrix} 
	$ 
	is invertible up to 
	$\left( 
	\begin{pmatrix}
	0 & 0 \\
	0 & P^{\prime} 
	\end{pmatrix} 
	,
	\begin{pmatrix}
	0 & 0 \\
	0 & Q^{\prime} 
	\end{pmatrix} 
	\right).$
	Hence, by applying Lemma \ref{r12 l12}, we deduce the statements in the corollary. 
\end{proof}

We end this section with a few more technical results which we will need in the rest of the paper. 

\begin{lemma} \label{r12 l14}
Let $N , M$ be a closed subspaces of $H$ such that $P_{N} , P_{M} \in \mathcal{A} $ and $D \in \mathcal{A}$ such that $D$ has the matrix 	$ 
\begin{pmatrix}
	D_{1} & D_{2} \\
	D_{1} & D_{4}
	\end{pmatrix} $
with respect to the decomposition $ N \oplus N^{ \perp} \rightarrow M \oplus M^{ \perp} ,$ where $ D_{1} $ is an isomorphism. If $S$ is the operator with matrix $ 
\begin{pmatrix}
	D_{1} ^{ -1} & 0 \\
	0 & 0
\end{pmatrix} $
with respect to the  decomposition $ M \oplus M^{ \perp} \rightarrow N \oplus N^{ \perp} , $then $ S \in \mathcal{A} .$

\end{lemma}

\begin{proof}
	We have that $ 
	\begin{pmatrix}
		D_{1}  & 0 \\
		0 & 0
	\end{pmatrix} 
= P_{M} D P_{N}\in \mathcal{A} . $  Let $U$ be the partial isometry from the polar decomposition of $D_{1} . $ The operator $ \tilde U $ given by the operator matrix $ \begin{pmatrix}
	U & 0 \\
	0 & 0
\end{pmatrix} $  with respect to the decomposition $ N \oplus N^{ \perp} \rightarrow M \oplus M^{ \perp} $ is obviously the partial isometry from the polar decomposition of the operator $
\begin{pmatrix}
	D_{1}  & 0 \\
	0 & 0
\end{pmatrix} ,$ hence $ \tilde U \in \mathcal{A} .$ Since $ D_{1} $ is an isomorphism, then $ \vert D_{1} \vert $ is invertible in $B(M) .$ Now, $ \vert 
\begin{pmatrix}
	D_{1} & 0 \\
	0 & 0
\end{pmatrix} \vert = 
\begin{pmatrix}
 \vert	D_{1} \vert  & 0 \\
	0 & 0
\end{pmatrix} ,$ so $
\begin{pmatrix}
	\vert	D_{1} \vert  & 0 \\
	0 & 0
\end{pmatrix} \in \mathcal{A} .$ Hence $
\begin{pmatrix}
	\vert	D_{1} \vert  & 0 \\
	0 & 0
\end{pmatrix} + P_{M^{ \perp }} = 
\begin{pmatrix}
	\vert	D_{1} \vert  & 0 \\
	0 & 1
\end{pmatrix} \in \mathcal{A} .$ The operator $
\begin{pmatrix}
	\vert	D_{1} \vert  & 0 \\
	0 & 1
\end{pmatrix} $ is positive, invertible operator in $ \mathcal{A} $ with its inverse $ \begin{pmatrix}
\vert	D_{1} \vert ^{ -1 } & 0 \\
0 & 1
\end{pmatrix} .$ This follows from the functional calculus. Hence $ \begin{pmatrix}
\vert	D_{1} \vert ^{ -1 } & 0 \\
0 & 1
\end{pmatrix} \in \mathcal{A} $ since $ \begin{pmatrix}
\vert	D_{1} \vert ^{ -1 } & 0 \\
0 & 1
\end{pmatrix} = P_{M} \begin{pmatrix}
\vert	D_{1} \vert  & 0 \\
0 & 1
\end{pmatrix} ^{ -1 } P_{M} .$ Next, notice that $ D_{1} ^{ -1 } = \vert D_{1} \vert ^{ -1 } U ^{*} .$ Hence $  \begin{pmatrix}
	D_{1}  ^{ -1 } & 0 \\
0 & 1
\end{pmatrix} =  \begin{pmatrix}
\vert	D_{1} \vert ^{ -1 } & 0 \\
0 & 1
\end{pmatrix} \tilde U ^{*} \in \mathcal{A} .$

\end{proof}
\begin{corollary} \label{isom}
	Let $N , M$ be a closed subspaces of $H$ such that $P_{N} , P_{M} \in \mathcal{A} $ and $D \in \mathcal{A}$ such that $D$ has the matrix 	$ 
	\begin{pmatrix}
		D_{1} & D_{2} \\
		D_{1} & D_{4}
	\end{pmatrix} $
	with respect to the decomposition $ N \oplus N^{ \perp} \rightarrow M \oplus M^{ \perp} .$ Then $D$ is invertible up to $ ( I -P_{N} , I - P_{M}) $ in $ \mathcal{A} $ if and only if $ D_{1} $ is an isomorphism. Moreover, $D$ is left invertible up to $ ( I -P_{N} ,  I- , P_{M}) $ in $\mathcal{A} $ if and only if $ D_{1} $ is bounded below, whereas $D$ is right invertible up to $ ( I -P_{N} , I - P_{M}) $ in $\mathcal{A} $ if and only if $ D_{1} $ is surjective. 
\end{corollary}
\begin{proof}
The first statement in the corollary follows from Lemma \ref{r12 l14} .\\
	Suppose next that $ D_{1} $ is bounded below. Then  $  P_{M} D  P_{N} $ is also bounded below since this operator has the matrix $
	\begin{pmatrix}
		D_{1} & 0 \\
		0 & 0
	\end{pmatrix} $ with respect to the decomposition  $N \oplus N^{ \perp} \rightarrow M \oplus M^{ \perp} .$ Let $ \tilde M = Im D_{1} ,$  which is equal to $ Im P_{M} D P_{N}  .$ Then $ \tilde M \subseteq M$ and $ P_{ \tilde M} \in Proj(\mathcal{A} ) .$ The operator  $ P_{ \tilde M}D_{1} $ is an isomorphism  from $ N$  onto $ \tilde M , $  and with respect to the decomposition $ N \oplus N^{ \perp} \rightarrow \tilde M \oplus  \tilde M^{ \perp} ,$ the operator $D$ has the matrix $
\begin{pmatrix}
P_{ \tilde M} D_{1} & 0 \\
0 & 0
\end{pmatrix} ,$ where $ P_{ \tilde M} D_{1} $ is an isomorphism. If $ S $ is the operator with the matrix $ \begin{pmatrix}
(P_{ \tilde M} D_{1} ) ^{ -1} & 0 \\
0 & 0
\end{pmatrix} $ with respect to the decomposition $  \tilde M \oplus  \tilde M^{ \perp} \rightarrow N \oplus N^{ \perp}  ,$ then $ S \in \mathcal{A} $ by Lemma \ref{r12 l14} . It is not hard to see that $S$ is left $ ( I -P_{N} , I - P_{M}) $ inverse of $D$ in $ \mathcal{A} .$ \\
If $ D_{1} $ is surjective, then by passing to the adjoints and using the previous arguments, we deduce the last statement in the corollary. 
	
\end{proof}

\begin{corollary} \label{projleft}
	Let $ P,Q \in Proj ( \mathcal{A} ) .$ Then $ I-P $ is invertible up to $ ( Q , P) $ if and only if $ H = ( I-Q ) (H) \tilde \oplus P(H) .$ 
\end{corollary}

\begin{proof}
By Corollary \ref{isom} we have that $ I-P $ is invertible up to  $ ( Q , P) $ if and only if $ I-P $ has the matrix $ \begin{pmatrix}
	R_{1} & R_{2}  \\
	0 & 0
\end{pmatrix} $with respect to the decomposition $$ ( I - Q) (H)  \oplus Q(H) \rightarrow ( I - P) (H)  \oplus P(H) $$ where $ R_{1} $ is an isomorphism. However, it is not hard to see that this is satisfied if and only if $ H = ( I - Q) (H) \tilde \oplus P(H) .$	
\end{proof}

\section{Spectral Fredholm theory in von Neumann algebras} 

In this section, for fixed $S,T \in \mathcal{A} $ we consider $M_{C} \in M_{2} (\mathcal{A})$ given by 

$ M_{C}=
\begin{pmatrix}
T & C \\
0 & S 
\end{pmatrix} 
$ where $C$ varies over $ \mathcal{A}.$ 
Set $$\sigma_{ef}(M_{C}) = \lbrace \lambda \in \mathbb{C} \text{ } \vert \text{ } M_{C}-\lambda I \text{  is not  }   \mathcal{A}- \text{Fredholm} \rbrace.$$ 
Similarly we define $\sigma_{ef}(T) $ and $\sigma_{ef}(S) $ for $T,S \in \mathcal{A} .$ By applying Corollary \ref{r12 c14}, we may in a similar way as in the proof of \cite[Proposition 3.1]{IS6} show that $\sigma_{ef}(M^{C}) \subseteq   \sigma_{ef}(T) \cup    \sigma_{ef}(S) .$

Now we are ready to give the main result in this section, which is a generalization of the result by \DJ{}or\dj{}evi\'{c} in \cite{DDj} in the setting of Fredholm operators in von Neumann algebras.

\begin{proposition} \label{r12 p16}
	Let $ \mathcal{A}$ be a properly infinite von Neumann algebra and $T,S \in \mathcal{A} .$ If there exists some $C \in \mathcal{A} $ such that the inclusion $\sigma_{ef}(M_{C}) \subset \sigma_{ef}(T) \cup \sigma_{ef}(S)  $ is proper, then $\sigma_{ef}(T) \cup \sigma_{ef}(S) = \sigma_{ef}(M_{C}) \cup (\sigma_{ef}(T) \cap \sigma_{ef}(S) )  .$
\end{proposition}

\begin{proof}
	As in the proof of \cite[Theorem 3.2]{IS6}, we write $M_{C} $ as $M_{C}=S^{\prime}C^{\prime}T^{\prime}$ where 
	$S^{\prime}=
	\begin{pmatrix}
	1 & 0 \\
	0 & S 
	\end{pmatrix} 
	,$
	$C^{\prime}=
	\begin{pmatrix}
	1 & C \\
	0 & 1 
	\end{pmatrix}
	$  
	and 
	$T^{\prime}=
	\begin{pmatrix}
	T & 0 \\
	0 & 1 
	\end{pmatrix}
	.$  
	Assume that $M_{C} $ is $\mathcal{A} -$Fredholm and let $P,Q \in Proj_{0}(M_{2} (\mathcal{A})) $ such that $M_{C} $ is invertible up to $(P, Q).$ By \cite[Corollary 6]{BJMA}, there is some $R \in Proj (M_{2} (\mathcal{A}))$ such that $ T^{\prime}$ is invertible up to $(P, R), S^{\prime}C^{\prime}$ is invertible up to $(R, Q)$ and $(I-R)T^{\prime}(I-P)=T^{\prime}(I-P) .$ Since $C^{\prime} $ is invertible, we have that $\tilde{R} \sim R ,$ where $\tilde{R} $ is the orthogonal projection onto $C^{\prime}R(H^{2}) .$ Moreover, $H^{2}=C^{\prime} (I-R)(H^{2}) \tilde{\oplus} C^{\prime} R(H^{2}) .$ If $\tilde{R}^{\prime} $ denotes the orthogonal projection onto $C^{\prime} (I-R)(H^{2})^{\perp} ,$ then obviously $\tilde{R}^{\prime} $ maps $C^{\prime} R(H^{2}) $ isomorphically onto $\tilde{R}^{\prime} (H^{2}) .$ Thus $\tilde{R}^{\prime} \sim \tilde{R} \sim R .$ Now, $C^{\prime} $ is invertible up to $( R,\tilde{R}^{\prime} ) $ and $(I-\tilde{R}^{\prime}) C^{\prime} (I-R) = C^{\prime} (I-R).$ Since $S^{\prime} C^{\prime} $ is invertible up to $(R, Q),$  we can deduce that $S^{\prime} $ is invertible up to $(\tilde{R}^{\prime} , Q) .$ Indeed, let $B$ be $(R, \tilde{R}^{\prime})$-inverse of $C^{\prime}$ and $\tilde{B}$ be $(R,Q)$-inverse of $S^{\prime}C^{\prime}.$ Then we get 
	$$C^{\prime}(I-R)\tilde{B}(I-Q)S^{\prime}(I-\tilde{R}^{\prime}) = 
	C^{\prime}(I-R)\tilde{B}(I-Q)S^{\prime}(I-\tilde{R}^{\prime})C^{\prime}(I-R)B= $$ 
	$$C^{\prime}(I-R)\tilde{B}(I-Q)S^{\prime} C^{\prime} (I-R)B= C^{\prime} (I-R)B =
	(I-\tilde{R}^{\prime})  C^{\prime} (I-R)B = I-\tilde{R}^{\prime} ,$$
	and 
	$$(I-Q)S^{\prime}(I-\tilde{R}^{\prime})C^{\prime}(I-R) \tilde{B}=
	(I-Q)S^{\prime}C^{\prime} (I-R) \tilde{B} = (I-Q) ,$$ so $ C^{\prime}(I-R)\tilde{B} $ is an $(\tilde{R}^{\prime} , Q)-$inverse of $ S^{\prime} .$   
	Hence, in particular we have that $T^{\prime} $ is left invertible up to $P$ and $S^{\prime}$ is right invertible up to $Q$.
	If we write $P $ and $Q$ as 
	$P=
	\begin{pmatrix}
	P_{1} & P_{2} \\
	P_{3} & P_{4} 
	\end{pmatrix}
	,$
	$Q=
	\begin{pmatrix}
	Q_{1} & Q_{2} \\
	Q_{3} & Q_{4} 
	\end{pmatrix}
	,$
	where $P_{j},Q_{j} \in \mathcal{A} $ for $ j \in \lbrace 1, \dots 4 \rbrace,$ then it follows that $FT=1-P_{1} $ and $SD=1-Q_{4} $ for some operators $F,D \in \mathcal{A} .$ By Corollary \ref{r12 c14} we have that $P_{1}, Q_{4} \in \mathfrak{m} ,$
	hence, by \cite[Lemma 10]{BJMA} and Corollary \ref{r12 c17}  we deduce that $T$ and $S$ are upper semi$- \mathcal{A} - $Fredholm and lower semi$- \mathcal{A} -$Fredholm, respectively.\\
	
	 If $p,q \in Proj_{0} (\mathcal{A})$ and $ r,r^{\prime} \in Proj (\mathcal{A}) $ such that $T$ invertible up to $(p,r)$ and $S$ is invertible up to $(r^{\prime},q),$ then, obviously, $T^{\prime} $ is invertible up to 
	$ \left( 
	\begin{pmatrix}
	p & 0 \\
	0 & 0 
	\end{pmatrix}
	,
	\begin{pmatrix}
	r & 0 \\
	0 & 0 
	\end{pmatrix}
	\right)$
	and $ S^{\prime}$ is invertible up to 
	$\left( 
	\begin{pmatrix}
	0 & 0 \\
	0 & r^{\prime} 
	\end{pmatrix}
	,
	\begin{pmatrix}
	0 & 0 \\
	0 & q 
	\end{pmatrix}
	\right).$ 
	By Lemma \ref{r12 l12} it follows that 
	$ \left( 
	\begin{pmatrix}
	p & 0 \\
	0 & 0 
	\end{pmatrix}
	,
	\begin{pmatrix}
	0 & 0 \\
	0 & q 
	\end{pmatrix}
	\right)
	\in Proj_{0} (M_{2}(\mathcal{A})) .$ 
	Hence, by Corollary \ref{r8 cor2.4}, we deduce that there exist projections $ E, \tilde{E}, E^{\prime}, \tilde{E}^{\prime} , L  , \tilde{L}  , L^{\prime}  ,   \tilde{L}^{\prime}$ in $\mathcal{A}$ such that $\tilde{E}  , \tilde{E}^{\prime} , \tilde{L}  , \tilde{L}^{\prime} $ are finite, $E \tilde{E} = E^{\prime}  \tilde{E}^{\prime}  = L \tilde{L} = L^{\prime}  \tilde{L}^{\prime}=0, $ $R \sim E,$ 
	$
	\begin{pmatrix}
	r & 0 \\
	0 & 0 
	\end{pmatrix}
	\sim E^{\prime}, $ $\tilde{R}^{\prime} \sim L, $ 
	$
	\begin{pmatrix}
	0 & 0 \\
	0 & r^{\prime} 
	\end{pmatrix}
	\sim L^{\prime}, $ 
	$  E  + \tilde{E}  \sim E^{\prime} +  \tilde{E}^{\prime} $ and $ L  + \tilde{L}  \sim L^{\prime} +  \tilde{L}^{\prime}.$

	Suppose now that $T$ is $\mathcal{A}-$Fredholm. Then by Lemma \ref{l 02} and by (\ref{f1}), we must have that $ r \in Proj_{0}(\mathcal{A}) .$ Hence, by Lemma \ref{r12 l12} we get that 
	$
	\begin{pmatrix}
	r & 0 \\
	0 & 0 
	\end{pmatrix}
	$ is finite, so $E^{\prime} \in Proj_{0}(M_{2} (\mathcal{A})) ,$ as 
	$ E^{\prime} \sim
	\begin{pmatrix}
	0 & 0 \\
	0 & r 
	\end{pmatrix}
	,$
	which gives that $E^{\prime} + \tilde{E}^{\prime} $ is finite. Therefore, $ E  + \tilde{E}  \in Proj_{0}(M_{2} (\mathcal{A})),$ hence $E \in Proj_{0}(M_{2} (\mathcal{A})) $ since $\tilde{E} $ is finite. However, $L \sim \tilde{R}^{\prime} \sim R \sim E ,$ so we get that $L + \tilde{L} $ is finite since $L, \tilde{L} \in Proj_{0}(M_{2} (\mathcal{A})) .$ 
	Thus, $L^{\prime} + \tilde{L}^{\prime} \in Proj_{0}(M_{2} (\mathcal{A})) $ since $ L + \tilde{L} \sim L^{\prime} + \tilde{L}^{\prime} ,$ so we must have that $L^{\prime} \in Proj_{0}(M_{2} (\mathcal{A})) $ because $\tilde{L}^{\prime} \in Proj_{0}(M_{2} (\mathcal{A})) .$ It follows that 
	$
	\begin{pmatrix}
	0 & 0 \\
	0 & r^{\prime} 
	\end{pmatrix}
	\in Proj_{0}(M_{2} (\mathcal{A}))$ 
	as 
	$
	\begin{pmatrix}
	0 & 0 \\
	0 & r^{\prime} 
	\end{pmatrix} 
	\sim L^{\prime} ,$ 
	hence, by Lemma \ref{r12 l12} we obtain that $r^{\prime} \in Proj_{0}(\mathcal{A}) .$ Since $S$ is invertible up to $(r^{\prime}, q) ,$ we get that $S$ is $\mathcal{A}-$Fredholm. Similarly we can show that $T$ is $\mathcal{A} -$Fredholm if $S$ is $\mathcal{A} -$Fredholm. If now $\lambda \in \mathbb{C} ,$ we can apply previous arguments to deduce that if $M_{C}- \lambda I $ and $ T - \lambda 1$ are $\mathcal{A}- $Fredholm, then $S-\lambda1$ is $\mathcal{A}- $Fredholm, and, similarly, if $M_{C} - \lambda I $ and $S-\lambda1 ,$ are $\mathcal{A} -$Fredholm, then $T-\lambda 1 $ is $\mathcal{A} -$Fredholm. This is because 
	$ 
	M_{C} - \lambda I =
	\begin{pmatrix}
	T - \lambda 1 && 0 \\
	0 & S - \lambda 1 
	\end{pmatrix} 
	,$ 
	so we can apply the previous arguments for arbitrary $ \lambda \in \mathbb{C} .$ Hence we can deduce that $$ ( \sigma_{ef}(T) \setminus \sigma_{ef}(S) ) \setminus \sigma_{ef} (M_{C}) = \varnothing , ( \sigma_{ef}(S) \setminus  \sigma_{ef}(T) ) \setminus \sigma_{ef} (M_{C}) = \varnothing  ,$$ which gives that $\sigma_{ef}(T) \cup \sigma_{ef}(S) = ( \sigma_{ef}(T) \cap \sigma_{ef}(S) ) \cup  \sigma_{ef} (M_{C}). $
	\end{proof}

Next, we wish to consider isolated points of the spectra of operators in $\mathcal{A} .$ We wish to show that if $0$ is an isolated point of the spectrum of an $\mathcal{A}-$ Fredholm operator, then the corresponding spectral projection is finite. To this end, we give first the following proposition. 

\begin{proposition} \label{r12 p17}
	Let $F$ be $\mathcal{A}-$ Fredholm and $P_{0} $ be some skew or orthogonal projection in $\mathcal{A} .$ Suppose that $F$ has the matrix 
	$\begin{pmatrix}
	 F_{1} 	& 0 \\
	 0 & F_{4} 
	\end{pmatrix} 
$ 
with respect to the decomposition $ H= \ker P_{0} \tilde \oplus Im P_{0} ,$ where $ F_{1} $ is an isomorphism. If $ P_{0} $ is not a finite operator, then $ P_{ Im P_{0} } F  P_{ Im P_{0} } \in \mathcal{K} \Phi ( P_{ Im P_{0} } \mathcal{A}  P_{ Im P_{0} } ) .$
\end{proposition}

\begin{proof} 
	Since $F$ is $\mathcal{A}-$ Fredholm, there exist some orthogonal 
	projections $ \tilde P , \tilde Q $ in $\mathcal{A} $ such that $I- \tilde P$ and $ I - \tilde Q $ are finite and such that $F$ is invertible up to $ ( I - \tilde P, I - \tilde Q ) .$ By \cite[Proposition 2.8]{KL} we may without loss of generality assume that $ (I-\tilde Q) F \tilde P = 0 .$ We have 
	$ \tilde P (H) ^ {\perp }  \oplus ( Im P_{0} \cap \tilde P (H) ) = \ker ( ( I - P_{0} ) \tilde P ) .$ Hence $ P_{\tilde P (H) ^ {\perp }  \oplus ( Im P_{0} \cap \tilde P (H) )} \in \mathcal{A} . $ Since $ Im P_{0} \cap \tilde P (H) = Im ( ( I - \tilde P )  P_{\tilde P (H) ^ {\perp }  \oplus ( Im P_{0} \cap \tilde P (H) )} ) ,$ we get that $ P_{Im P_{0} \cap \tilde P (H) } \in \mathcal{A} .$\\
	
	 Let $N$ be the orthogonal complement of  $Im P_{0} \cap \tilde P (H) $ in $ Im P_{0} .$ Then $ P_{N} = P_{ Im P_{0} } - P_{ Im P_{0} \cap \tilde P (H) } \in \mathcal{A} .$ Now, $ I- \tilde P $ is injective on $ N ,$ hence $ \ker ( ( I - \tilde P ) P_{N} ) ^{ \perp } = N .$ Therefore,  $ P_{N} \sim P_{ \overline{ Im ( I - \tilde P ) P_{N} } } .$ Since $ ( I - \tilde P ) P_{N} $ is a finite operator because $ ( I - \tilde P ) $ is a finite operator, we get that $ P_{ \overline{ Im ( I - \tilde P ) P_{N} } } \in Proj_{0} ( \mathcal{A} ) .$ Hence $ P_{N} \in Proj_{0} ( \mathcal{A} ) .$\\
	 
	  Notice that, since $ (I - \tilde Q ) F \tilde P = 0 $ and $F$ is invertible up to $ ( I - \tilde P , I - \tilde Q ) ,$ we have that $F$ maps $ Im \tilde P $ isomorphically onto $ F ( \tilde P (H) ) = \tilde Q (H) .$ It follows that $F$ maps $ Im P_{0} \cap \tilde P (H)$ isomorphically onto $ F (Im P_{0} \cap \tilde P (H)) ,$ so $ F (Im P_{0} \cap \tilde P (H)) $ is closed. Since $ F (Im P_{0} \cap \tilde P (H)) = Im F P_{Im P_{0} \cap \tilde P (H)} ,$ we have that $ P_{ F (Im P_{0} \cap \tilde P (H))} \in \mathcal{A} .$  If $M$ denotes the orthogonal complement of $  F (Im P_{0} \cap \tilde P (H)) $ in $ Im P_{0} ,$ then, since $ P_{M} = P_{ Im P_{0} } - P_{ F (Im P_{0} \cap \tilde P (H))} ,$ we have that $ P_{M} \in \mathcal{A} .$ \\
	Observe now that $F$ has the matrix $\begin{pmatrix}
		F_{1} 	& F_{2} \\
		0 & F_{4} 
	\end{pmatrix} 
	,$ with respect to the decomposition $  ( \ker P_{0} \tilde \oplus (Im P_{0} \cap \tilde P (H)) ) \tilde \oplus N \rightarrow  ( \ker P_{0} \tilde \oplus F(Im P_{0} \cap \tilde P (H)) ) \tilde \oplus M ,$ where $ F_{1} $ is an isomorphism. Set $$ \tilde N = \ker P_{0} \tilde \oplus (Im P_{0} \cap \tilde P (H)) ,$$  $$ \tilde M =  \ker P_{0} \tilde \oplus F(Im P_{0} \cap \tilde P (H)) .$$ Then, since $ \tilde M = Im ( I - P_{0} + P_{ F(Im P_{0} \cap \tilde P (H))} P_{0} )$ and $ \tilde N = Im ( I - P_{0} + P_{ Im P_{0} \cap \tilde P (H)} P_{0} ) ,$ we have that $ P_{ \tilde M} ,  P_{ \tilde N} \in \mathcal{A} .$ Hence $ P_{ \tilde M ^{ \perp} } ,  P_{ \tilde N^{ \perp } } \in \mathcal{A} .$ Since $ H = \tilde N \tilde \oplus N ,$ we have that $ P_{ \tilde N^{ \perp } } $ is injective on $N$ and $ \tilde N^{ \perp} = P_{ \tilde N^{ \perp } } (N) .$ Hence, we get that $ \ker P_{ \tilde N^{ \perp } } P_{N} = N^{ \perp } $ and $ Im P_{ \tilde N^{ \perp } } P_{N} = \tilde N^{ \perp} .$ Therefore, $ P_{N} \sim P_{ \tilde N^{ \perp } } ,$ so $ P_{ \tilde N^{ \perp } } \in Proj_{0} ( \mathcal{A} ) .$ Likewise, $ P_{M} \sim P_{ \tilde M^{ \perp } } .$ Now, since $F$ maps $ \tilde N $ isomorphically onto $ \tilde M ,$ then by Lemma \ref{r12 l14} we have that $F$ is invertible up to $ ( P_{ \tilde N^{ \perp } } , P_{ \tilde M^{ \perp } }  ) $ in $ \mathcal{A} .$ Since $F$ is $ \mathcal{A} -$ Fredholm and $  P_{ \tilde N^{ \perp } } \in Proj_{0} ( \mathcal{A} ) $, by Lemma \ref{l 02} and (\ref{f1}) we must have that $ P_{ \tilde M^{ \perp } } \in Proj_{0} ( \mathcal{A}) .$ Thus, $ P_{M} \in Proj_{0} ( \mathcal{A}) .$ \\
	Consider next the von Neumann algebra $ P_{ Im P_{0 } } \mathcal{A} P_{ Im P_{0 } } .$ If $ P_{ Im P_{0 } } $ is not finite, then $ P_{ Im P_{0 } } \mathcal{A} P_{ Im P_{0 } } $ is also a properly infinite von Neumann algebra. Now, $ P_{M} , P_{N} \in Proj_{0} ( P_{ Im P_{0 } }  \mathcal{A} P_{ Im P_{0 } } ) .$ Since $ P_{Im P_{0} \cap \tilde P (H) } = P_{ Im P_{0} } - P_{N} ,  P_{ F (Im P_{0} \cap \tilde P (H) ) } = P_{ Im P_{0} } - P_{M} $ and $F$ maps $ Im P_{0} \cap \tilde P (H) $ isomorphically onto $ F (Im P_{0} \cap \tilde P (H)) ,$ it follows by Corollary \ref{isom} that $ P_{ Im P_{0} } F P_{ Im P_{0} } ,$ which is equal to $ F P_{ Im P_{0} } ,$ is invertible up to $ ( P_{N} , P_{M} ) $ in $ P_{ Im P_{0 } } \mathcal{A} P_{ Im P_{0 } } .$ Hence $ P_{ Im P_{0} } F P_{ Im P_{0} } \in \mathcal{K} \Phi (P_{ Im P_{0 } } \mathcal{A} P_{ Im P_{0 } }). $
	\end{proof} 

We can deduce now the desired result concerning spectral projections as a corollary of Proposition \ref{r12 p17}.  
\begin{corollary}\label{spectralprojection}
	Let $ F \in \mathcal{A} $ and $\alpha$ be an isolated point of  $ \sigma (F).$ If $ F - \alpha I $ is $ \mathcal{A} -$ Fredholm and $ P_{0} $ is the spectral projection corresponding to $ \alpha ,$ then $ P_{0} $ is finite operator. 
	\end{corollary}
\begin{proof}
	Note that $ \sigma (F) $ in $ \mathcal{A} $ is the same as the spectrum of $F$ in $B(H)$ since $ \mathcal{A} $ is a von Neumann algebra. By \cite[Section 3]{ZZRD} it follows that $F - \alpha I$ satisfies the conditions of Proposition \ref{r12 p17} with respect to the decomposition $ \ker P_{0} \tilde \oplus Im P_{0} = H .$ Moreover, $ F - \lambda I $ maps $ Im P_{0} $ isomorphically onto $ Im P_{0} $ for all $ \lambda \neq \alpha .$ If $ P_{ Im P_{0} } $ is not a finite projection, by Proposition \ref{r12 p17} we have that $ P_{ Im P_{0} } (F - \lambda I) P_{ Im P_{0} } \in \mathcal{K} \Phi (P_{ Im P_{0 } } \mathcal{A} P_{ Im P_{0 } }) $ for all $ \lambda \in \mathbb{C} .$ By the similar arguments as in the proof of \cite[Corollary 2.8]{ZZRD}  we can deduce that $ P_{ Im P_{0 } } \mathcal{A} P_{ Im P_{0 } } $ consists only of finite operators, so $ P_{ Im P_{0 } } $ is a finite operator, which contradicts the assumption in the beginning of this proof that $  P_{ Im P_{0 } } $ is not finite.   
\end{proof} 
\begin{remark} \label{rem101}
	We notice that by the arguments from the proof of Proposition \ref{r12 p17} it also follows that if  $N , M$ are closed subspaces of $H$ such that $P_{N} , P_{M} \in \mathcal{A} ,$ then $ P_{ M \cap N} \in \mathcal{A} .$ 
\end{remark}

The theory regarding isolated points of the spectrum of Fredholm operators on Hilbert and Banach spaces is closely connected to the concept of Browder operators, as illustrated in \cite[Theorem 3.1]{ZZRD}.  Motivated by \cite[Definition 5.7]{IS13} we give now the following definition of generalized $\mathcal{A}$-Browder operators. 

\begin{definition}
	Let $F\in\mathcal{A}. $ We say that $ F $ is generalized $\mathcal{A}$-Browder if there exists a decomposition 
	$$ H = M \tilde \oplus N \stackrel{F}{\longrightarrow} M \tilde \oplus N= H  $$ 
	with respect to which $ F $ has the matrix 
	$\begin{pmatrix}
		F_{1}  & 0 \\
		0 & F_{4}
	\end{pmatrix} ,$ 
	where $ F_{1} $ is an isomorphism and such that $ P_{N} \in Proj_{0} (\mathcal{A}).$ 
\end{definition}

We have the following lemma.

\begin{lemma}\label{Browder}
	Let $F\in\mathcal{A}. $ If $ F $ is  generalized $\mathcal{A}$-Browder, then $ F $ is  $\mathcal{A}$-Fredholm.
	
\end{lemma}
\begin{proof}
	Let 
	$$ H = M \tilde \oplus N \stackrel{F}{\longrightarrow} M \tilde \oplus N= H  $$ 
	be an $\mathcal{A}$-Browder decomposition for $ F .$ By the proof of \cite[Lemma 2.5]{IS10}, $ F $ has the matrix
	$\begin{pmatrix}
		F_{1}  & 0 \\
		0 & F_{4}
	\end{pmatrix} $ 
	with respect to the decomposition
	$$ H =N^{\perp}  \oplus N  \stackrel{F}{\longrightarrow} F(N^{\perp}) \tilde \oplus N= H , $$ 
	where $ F_{1} $ is an isomorphism. Hence, $F(N^{\perp}) $ is closed and $ F $ has he matrix 
	$\begin{pmatrix}
		F_{1}  & \tilde{F_{2}} \\
		0 &\tilde{F_{4}}
	\end{pmatrix} ,$
	with respect to the decomposition 
	$$ H =N^{\perp} \oplus N  \stackrel{F}{\longrightarrow} F(N^{\perp}) \tilde \oplus F(N^{\perp})^{\perp} = H  .$$  
	Now, $F(N^{\perp}) = Im F(I-P_{N}) ,$ so $P_{F(N^{\perp})} \in \mathcal{A} $ since $ P_{N}\in\mathcal{A}. $ By Corollary \ref{isom} we deduce that $ F $ is invertible up to $(
	P_{N}, I-P_{F(N^{\perp})} )$ in $\mathcal{A}.$ Since $H= F(N^{\perp}) \tilde{ \oplus } N $, it follows that $I-P_{F(N^{\perp})} $ maps $ N $ isomorphically onto $ F(N^{\perp})^{\perp}. $ Thus we get that $ F(N^{\perp})^{\perp} = Im (I-P_{F(N^{\perp})})P_{N} $ and $N^{\perp}= \ker ((I-P_{F(N^{\perp})})P_{N}) ,$ which gives that $ P_{N} \sim  P_{F(N^{\perp})^{\perp}}.$ Hence $ I-P_{F(N^{\perp})} \in Proj_{0} (\mathcal{A}) $ because $ P_{N} \in Proj_{0} (\mathcal{A}) ,$ so $ F $ is $\mathcal{A}$-Fredholm by Lemma \ref{r12 l11}.
	\end{proof}

The next corollary is motivated by \cite[Theorem 3.1]{ZZRD}.    
\begin{corollary}
	Let $F\in\mathcal{A}$ and suppose that $ 0 $ is an isolated point of $ \sigma (F). $ Then $ F $ is $\mathcal{A}$-Fredholm if and only if $ F $ is generalized $\mathcal{A}$-Browder.
\end{corollary}
\begin{proof}
	The implication in one direction follows from Corollary \ref{spectralprojection}, whereas the implication in the other direction follows from Lemma \ref{Browder}.
\end{proof}

We shall devote the rest of this section to generalizations in the setting of operators in von Neumann algebras of the results by Zemanek given in \cite{ZE} concerning the relationship between the spectra of an operator and of its compressions. To this end, we give first the following auxiliary technical lemma.
	\begin{lemma}  \label{r10l 1.8}
	Let $p \in \mathcal{A} $ be a projection and $a \in \mathcal{A} .$ Then $a$ is invertible up to $(p,p)$ if and only if $a$ is invertible up to pair $(p, q)$ where $q$ is projection satisfying that $qa(1-p)=0 $ and that $1-p $ is invertible up to pair $(q, p).$ 
\end{lemma}

\begin{proof}
	Suppose that $a$ is invertible up to $(p, p)$ and let $b^{\prime} $ be $(p, p)$ inverse of $a.$ By \cite[Proposition 2.8]{KL} there exists some projection $q \in \mathcal{A} $ such that $a$ is invertible up to $(p, q),$ $ qa(1-p)=0$ and $q \sim p .$ Then we have 
	$$(1-p)(1-q)a(1-p)b^{\prime} = (1-p) a (1-p) b^{\prime} = (1-p)$$  
	and 
	$$ a(1-p)b^{\prime}(1-p)(1-q) = a (1-p) b^{\prime} (1-p) a (1-p) b $$ 
	where $ b$ is $(p, q)$ inverse of $a.$ However, we have 
	$$a (1-p) b^{\prime} (1-p) a (1-p) b = a (1-p) b = 1 -q ,$$ 
	so $a (1-p) b^{\prime} $ is $(q, p)$ inverse of $1-p ,$ which proves the implication in one direction. Assume now that $a$ is invertible up to some  pair $(p, q)$ where $qa(1-p)=0$ and $1-p $ is invertible up to pair $(q, p).$ Let $b $ be $(p, q)$ inverse of $a$ and $c $ be $(q, p)$ inverse of $1-p .$ Then we get that 
	$$bc (1-p) a (1-p) = bc (1-p) (1-q) a (1-p) = b (1-q) a (1-p) = 1-p $$ 
	and 
	$$(1-p) a (1-p) b c = (1-p) (1-q) a (1-p) b c = (1-p) (1-q) c = 1-p ,$$ 
	so $bc $ is $(p, p)$ inverse of $a,$ which proves the opposite implication.
\end{proof}
From the proof of Lemma \ref{r10l 1.8} we can deduce the following corollary. 

\begin{corollary} \label{r10c 1.12}
	Let $a \in \mathcal{A} $ and $p, q$ be projections in $ \mathcal{A} .$ Then $a$ is invertible up to $(p, q)$ if and only if there exists some projection $q^{\prime} $ in $\mathcal{A} $ such that $a$ is invertible up to $(p,q^{\prime}) ,$ $ q^{\prime}a(1-p)=0$  and $(1-q) $ is invertible up to $(q^{\prime}, q) .$ Similarly, $a$ is invertible up to $(p, q)$ if and only if there exists some projection $p^{\prime} $ in $\mathcal{A} $ such that $a$ is invertible up to $(p^{\prime}, q), $ $(1-q)ap^{\prime} = 0$ and $(1-p) $ is invertible up to $ (p^{\prime}, p).$
\end{corollary}

\begin{proof}
	The first statement follows from the proof of Lemma \ref{r10l 1.8} whereas the second statement follows from the first statement by passing to the adjoints and applying Lemma \ref{l 01}.
\end{proof}

Let now $\tilde{ \mathcal{K}\Phi}_{0}(\mathcal{A}) $ be the set of all $T \in \mathcal{A}$ for which there exist finite projections $P,Q \in \mathcal{A}$ such that $T$ maps $ (I-P)(H)$ isomorphically onto $(I-Q)(H)$ and such that $H=(I-Q) (H) \tilde{\oplus} P(H).$ For $F \in \mathcal{A}$ set 
$$\sigma_{eW}(F):= \lbrace \lambda \in \mathbb{C} \mid F-\lambda I \notin \tilde{ \mathcal{K}\Phi}_{0}(\mathcal{A}) \rbrace . $$
Then we have the following corollary.
\begin{corollary}
	The set $\tilde{ \mathcal{K}\Phi}_{0}(\mathcal{A})  $ is open in $\mathcal{A}$ and $\tilde{ \mathcal{K}\Phi}_{0}(\mathcal{A}) \subseteq \mathcal{K}\Phi_{0}(\mathcal{A}).$ Moreover, for $F \in \mathcal{A}$ we have 
	$$\sigma_{eW}(F) = \cap \lbrace \sigma (PF_{\mid_{ImP}}) \mid P \text{ is a finite projection in } \mathcal{A} \rbrace.$$
\end{corollary}
\begin{proof}
	The statements follow from Lemma \ref{r10l 1.8} , Corollary \ref{isom} ,Corollary \ref{projleft} and \cite[Lemma 2.4]{KL} .
\end{proof}

As a consequence of Corollary \ref{r10c 1.12} we obtain also the next two lemmas which characterize left and right invertibility in terms of upper triangular decompositions. 

\begin{lemma} \label{r10l 1.10}
	Let $p \in \mathcal{A} $ be a projection and $ a \in \mathcal{A} .$ Then $a$ is left invertible up to $(p, p)$ if and only if there exist projections $q, s$ in $\mathcal{A} $ such that $a$ is invertible up to pair $(p,q) , \text{ } qa(1-p)=0, \text{ } p \leq s$ and $1-s $ is invertible up to pair $(q, s).$
\end{lemma}

\begin{proof}
	Suppose that $a$ is left invertible up to $(p, p).$ By \cite[Lemma 2.5]{KL} there exists some projection $s$ in $\mathcal{A} $ such that $(1-s)(1-p)a(1-p) = (1-p) a (1-p) $ and $ (1-p) a (1-p) $ is invertible up to pair $(p,s) .$ Then we have 
	$$p (1-s) = p (1-s) (1-p) a (1-p) b = p (1-p) a (1-p) b = 0$$ 
	where $b $ is $(p, s)$ inverse of $(1-p) a (1-p) .$ Hence $(1-p) (1-s) = (1-s) (1-p) = 1-s ,$ so it follows that $a$ is invertible up to $(p, s)$ and $b$ is $(p, s)$ inverse of $a.$ By Corollary \ref{r10c 1.12} the implication in one direction follows.
	
	Conversely, let $q, s$ be projections in $\mathcal{A}$ such that $a$ is invertible up to $ (p,q),$ $ q a (1-p) =0, \text{ } p \leq s$ and $1-s $ is invertible up to $(q, s).$ Let $b$ be $ (p, q)$ inverse of $a$ and $c$ be $(q, s)$ inverse of $1-s .$ Then we get 
	$$bc (1-s)  a (1-p) = bc (1-s)  (1-q) a (1-p) = b (1-q) a (1-p) = 1-p ,$$ 
	so $a$ is left invertible up to $(p, p).$
\end{proof}

Set $ \tilde{ \mathcal{K}\Phi}_{+} ^{-}(\mathcal{A}) $ to be the set of all $ T \in \mathcal{A} $ for which there exist some $ P, Q , S \in  Proj (\mathcal{A} ) $ such that $T$ maps $ (I-P) (H) $ isomorphically onto $ (I-Q) (H), H = (I-Q) (H) \tilde \oplus S(H) , P(H) \subseteq S(H)  $ and $ P$ is a finite projection. Then $ \tilde{ \mathcal{K}\Phi}_{+} ^{-}(\mathcal{A}) \subseteq  \mathcal{K}\Phi_{+} ^{-}(\mathcal{A}) .$ For $F \in \mathcal{A}$ set 
$$\sigma_{eA}(F):= \lbrace \lambda \in \mathbb{C} \mid F-\lambda I \notin \tilde{ \mathcal{K}\Phi}_{+} ^{-}(\mathcal{A}) \rbrace . $$  

\begin{corollary}
	We have $$\sigma_{eA}(F) = \cap \lbrace \sigma_{a} (PF_{\mid_{ImP}}) \mid P \text{ is a finite projection in } \mathcal{A} \rbrace,$$
	where $ \sigma_{a} (PF_{\mid_{ImP}}) = \lbrace \lambda \in \mathbb{C}  \mid (PF-\lambda I )_{\mid_{ImP}} \text{ is not bounded below on } Im P \rbrace .$
\end{corollary}
\begin{proof}
	The statement follows from Lemma \ref{r10l 1.10} , Corollary \ref{isom} and Corollary \ref{projleft} .
\end{proof}

\begin{lemma} \label{r10l 1.11}
	Let $p \in \mathcal{A} $ be a projection and $a \in \mathcal{A} .$ Then $a$ is right invertible up to $(p, p)$ if and only if there exist projections $q, s$ in $\mathcal{A} $ such that $a$ is invertible up to $ (s,q),$ $qa(1-s)=0,$ $p \leq s$ and $(1-p) $ is invertible up to $(q, p).$
\end{lemma}

\begin{proof}
	Suppose that $a$ is right invertible up to $(p, p).$ Then $a^{*}$ is left invertible up to $(p, p),$ hence, by the proof of Lemma \ref{r10l 1.10} there exists some projection $s \geq p $ such that $a^{*}$ is invertible up to $(p, s).$ By Lemma \ref{l 01} it follows that $a$ is invertible up to $(s,p).$ Hence, by Corollary \ref{r10c 1.12} the implication in one direction follows. Conversely, let $q, s$ be projections in $\mathcal{A}$ such that $a$ is invertible up to $(s,q), $ $qa(1-s)=0,$ $p \leq s$ and $1-p $ is invertible up to $(q, p).$ Let $b $ be $(s, q)$ inverse of $a$ and $ c$ be $(q, p)$ inverse of $1-p .$ Then we get 
	$$(1-p)a(1-p)(1-s)bc = (1-p)a(1-s)bc $$ 
	$$= (1-p)(1-q)a(1-s)bc = (1-p)(1-q)c = 1-p  ,$$ 
	so $a$ is right invertible up to $(p, p).$
\end{proof}
Set $ \tilde{ \mathcal{K}\Phi}_{-} ^{+}(\mathcal{A}) $ to be the set of all $ T \in \mathcal{A} $ for which there exist some $ P, Q , S \in  Proj (\mathcal{A} ) $ such that $ T$ maps $ (I-S) (H) $ isomorphically onto $ (I-Q) (H), H = (I-Q) (H) \tilde \oplus P(H), P(H) \subseteq S(H) $ and $Q$ is a finite projection.\\ Notice that, since we have that $ H = (I-Q) (H) \tilde \oplus P(H) , $ it follows that $ Q $ maps  $P(H) $ isomorphically onto $ Q(H) ,$ hence $ P \sim Q .$ Moreover, $ \tilde{ \mathcal{K}\Phi}_{-} ^{+}(\mathcal{A}) \subseteq  \mathcal{K}\Phi_{-} ^{+}(\mathcal{A}).$ \\
For $F \in \mathcal{A}$ set 
$$\sigma_{eD}(F):= \lbrace \lambda \in \mathbb{C} \mid F-\lambda I \notin \tilde{ \mathcal{K}\Phi}_{-} ^{+}(\mathcal{A}) \rbrace . $$
\begin{corollary}
	We have $$\sigma_{eD}(F) = \cap \lbrace \sigma_{d} (PF_{\mid_{ImP}}) \mid P \text{ is a finite projection in } \mathcal{A} \rbrace,$$
	where $ \sigma_{d} (PF_{\mid_{ImP}}) = \lbrace \lambda \in \mathbb{C}  \mid (PF-\lambda I )_{\mid_{ImP}} \text{ is not  onto } Im P \rbrace .$
\end{corollary}
\begin{proof}
	The statement follows from Lemma \ref{r10l 1.11} ,Corollary \ref{isom} and Corollary \ref{projleft} . 
\end{proof}

\section{Semi-B-Fredholm theory in von Neumann algebras}

In this section we wish to introduce semi-B-Fredholm theory in von Neumann algebras as a generalization of semi-B-Fredholm theory on Hilbert and Banach spaces established in \cite{BS} and \cite{BM}.
First we recall the following definition regarding the concept of exactness in a von Neumann algebra. 

\begin{definition} \label{D15}   %    \textbf{{ D15}}\\
	 Let $M_{1},\dots ,M_{n}  $ be closed subspaces of $H$ such that $P_{M_{j} } \in Proj(\mathcal{A} ) $ for all $ j.$ We say that the sequence $ 0 \rightarrow M_{1} \rightarrow M_{2} \rightarrow \dots  \rightarrow M_{n} \rightarrow 0$ is exact if for each $k \in \lbrace 2,\dots ,n-1 \rbrace $ there exist closed subspaces $M_{k}^{\prime}$ and $ M_{k}^{\prime \prime}$ with $ P_{M_{k}^{\prime} } , P_{ M_{k}^{\prime \prime} } \in Proj(\mathcal{A} )$ such that the following holds:\\
	1)	$ M_{k}=M_{k}^{\prime} \tilde{ \oplus} M_{k}^{\prime \prime} $ for all $k \in \lbrace 2,\dots ,n-1 \rbrace ;$\\
	2)	$P_{M_{2}^{\prime} } \sim P_{M_{1} } $ and $ P_{M_{n-1}^{\prime \prime} } \sim P_{M_{n} };$\\
	3)	$ P_{M_{k}^{\prime \prime} } \sim P_{M_{k+1}^{\prime} } $ for all $k \in \lbrace 2,\dots ,n-2 \rbrace  . $
\end{definition}

The next lemma can be proved in a similar way as \cite[Proposition 3]{IS5} and \cite[Lemma 2]{IS5}. For the convenience, we give the full proof since we will refer to the certain parts of this proof later in the paper. 

\begin{lemma} \label{r12 l15}
	Let $ D, F \in \mathcal{A} $ and suppose that $ Im F, Im D $ and $ Im DF $ are closed. Then the sequence $$ 0 \rightarrow \ker F \rightarrow \ker DF \rightarrow \ker D \rightarrow ImF^{ \perp } \rightarrow Im DF^{ \perp } \rightarrow ImD^{ \perp} \rightarrow 0 $$ is exact in $ \mathcal{A} .$
\end{lemma}
\begin{proof}
	By Remark \ref{rem101} we have that $ P_{ \ker D \cap ImF} \in \mathcal{A} .$ Let $M$ and $W$ denote the orthogonal complements of $ \ker D \cap ImF $ in $ \ker D $ and $ ImF ,$ respectively. Then, since $ P_{M} = P_{ \ker D } - P_{ker D \cap ImF} $ and $ P_{W} = P_{ Im F } - P_{ker D \cap ImF} ,$ we get that $ P_{M} , P_{W} \in \mathcal{A} .$ Likewise, if $X$ denotes the orthogonal complement of $ Im DF $ in $ Im D ,$ then $ P_{X} \in \mathcal{A} .$\\
	With respect to the decomposition $ \ker D ^{ \perp} \oplus \ker D  \rightarrow ImD \oplus ImD ^{ \perp } ,$ the operator $D$ has the matrix $ 	\begin{pmatrix}
		D_{1}  & 0 \\
		0 & 0
	\end{pmatrix} ,$ where $ D_{1} $ is an isomorphism. Let $S$ be the operators with the matrix $ \begin{pmatrix}
	D_{1} ^{ -1} & 0 \\
	0 & 0
\end{pmatrix} $ with respect to the decomposition $ ImD \oplus ImD ^{ \perp } \rightarrow \ker D ^{ \perp} \oplus \ker D  .$ Then, by Lemma \ref{r12 l14} we have that  $ S \in \mathcal{A} .$ By the same arguments as in the proof of \cite[Proposition 3]{IS5} we deduce that $P_{ \ker D ^{ \perp }} $ maps $W$ isomorphically onto $ Im SDF $ and that $ H = ImF \tilde \oplus S(X) \tilde \oplus M .$ Moreover, since $ S(X) = Im S P_{X} $ is closed, we have that $ P_{S(X)} \in \mathcal{A} .$ Hence $ P_{ S(X) \oplus M} =  P_{S(X)} + P_{M} \in \mathcal{A} .$ From the equation $ H = ImF \tilde \oplus S(X) \tilde \oplus M  $ we obtain that $ Im P_{ ImF ^{ \perp}} = Im P_{ ImF ^{ \perp}} P_{ S(X) \oplus M} = ImF ^ { \perp } $ and $ \ker P_{ ImF ^{ \perp}} P_{ S(X) \oplus M} = (S(X) \oplus M) ^ { \perp } .$ Therefore, $ P_{ S(X) \oplus M} \sim P_{ ImF ^{ \perp}} $ in $ \mathcal{A} ,$ so $ (P_{ S(X) } + P_{M}) \sim P_{ ImF ^{ \perp}} .$ \\
Further, let $ \tilde W $ denote the orthogonal complement of $ \ker F $ in $ \ker DF .$ Then $ P_{ \tilde W} ,$ which is equal to $ P_{ \ker DF  } - P_{ \ker F  } ,$ belongs to $ \mathcal{A} .$ By the same arguments as in the proof of \cite[Proposition 3]{IS5} we can show that $F$ maps $ \tilde W $ isomorphically onto $ \ker D \cap ImF .$ Hence, since $ \ker F P_{ \tilde W} = \tilde W ^ { \perp } $ and $ Im F P_{ \tilde W} = \ker D \cap ImF ,$ we get that $ P_{ \tilde W} \sim P_{\ker D \cap ImF} $ in $ \mathcal{A} . $ \\
Next, since $S$ maps $X$ isomorphically onto $S(X) ,$ we get that $ P_{ S(X) } \sim P_{X} $ as $ Im S P_{X} =S(X) $ and $ \ker S P_{X} = X ^ { \perp } .$ \\
Summarizing all this, we obtain the following chain of relations: 
$$  P_{ \ker DF  } = P_{ \ker F  } + P_{ \tilde W } ,
 P_{ \tilde W} \sim P_{\ker D \cap ImF} .$$
 $$ P_{ker D \cap ImF} + P_{M} = P_{ \ker D }  , (P_{ S(X) } + P_{M}) \sim P_{ ImF ^{ \perp}} . $$ 
 $$ P_{ S(X) } \sim P_{X} , P_{ ImDF ^ { \perp }} = P_{ ImD ^ { \perp }} + P_{X} .$$ 
 \end{proof}

In a similar way as in the proof of \cite[Lemma 9]{BJMA} and \cite[Corollary 25]{BJMA}, we can prove the following lemma. 
\begin{lemma} \label{r12 l16}
	Let $F \in \mathcal{A}$ and $ P_{0} \in Proj ( \mathcal{A} )$ such that $ P $ is not finite. 
	Suppose that $ F P_{0} = P_{0} FP_{0} .$  Then the following holds.\\
	1) If $ F P_{0} \in \mathcal{K} \Phi_{+} (P_{0} \mathcal{A}  P_{0} ) ,$ then $ P_{ \ker F \cap P_{0} (H)} $ is finite. Hence, if $ Im F P_{0} $ is closed, then $ F P_{0} \in \mathcal{K} \Phi_{+} (P_{0} \mathcal{A}  P_{0} ) $ if and only if $ P_{ \ker F \cap P_{0} (H)} $ is finite.\\
	2) If $ F P_{0} \in \mathcal{K} \Phi_{-} (P_{0} \mathcal{A}  P_{0} ) ,$ then $ P_{0} - P_{ \overline{ Im F P_{0} } } $ is finite. Hence, if $ Im F P_{0} $ is closed, then $ F P_{0} \in \mathcal{K} \Phi_{-} (P_{0} \mathcal{A}  P_{0} ) $ if and only if $ P_{0} - P_{  Im F P_{0}  } $ is finite.
\end{lemma}

\begin{remark}
	If $P \in Proj(P_{0} \mathcal{A}  P_{0} ),$ then it is not hard to see that $ P \in Proj_{0} (P_{0} \mathcal{A}  P_{0} ) $ if and only if $ P \in Proj_{0} ( \mathcal{A}   ) .$ Indeed, the implication in one direction is obvious, whereas for the implication in the opposite direction it suffices to observe as in the proof of Lemma \ref{r12 l12} that if $ P, Q \in Proj(P_{0} \mathcal{A}  P_{0} )$ and $ V \in \mathcal{A} $ is such that $ P= V V^{*} $ and $ Q= V^{*} V ,$ then $ P= (PVQ) (PVQ)^{*} $ and $ Q= (PVQ)^{*} (PVQ) ,$ however, $ PVQ \in P_{0} \mathcal{A}  P_{0} .$
\end{remark}

\begin{lemma} \label{r12 l17}
	Let $ F \in \mathcal{A} $ and suppose that $ Im F$ and $ Im F ^ {2}$ are closed.Then the following holds.\\
	1) If  $ F  \in \mathcal{K} \Phi_{+} ( \mathcal{A}   ) ,$  then  $ F P_{Im F} \in \mathcal{K} \Phi_{+} (P_{Im F} \mathcal{A}  P_{Im F} ) .$ \\
	2) If  $ F  \in \mathcal{K} \Phi_{-} ( \mathcal{A}   ) ,$  then  $ F P_{Im F} \in \mathcal{K} \Phi_{-} (P_{Im F} \mathcal{A}  P_{Im F} ) .$ \\
	3) If  $ F  \in \mathcal{K} \Phi ( \mathcal{A}   ) ,$  then  $ F P_{Im F} \in \mathcal{K} \Phi (P_{Im F} \mathcal{A}  P_{Im F} ) $ and in this case $ index \text{ } F =  index \text{ } F P_{ Im F} .$
	\end{lemma} 
\begin{proof}
	We recall first that $ P_{ \ker F ^{ \perp}} \sim P_{Im F} .$ Hence, if $ P_{ \ker F } \in Proj_{0} ( \mathcal{A}   ) ,$ then $ P_{Im F} $ can not be finite since $ \mathcal{A} $ is properly infinite. Similarly, if $ P_{Im F ^{ \perp }} $ is finite, then $ P_{Im F} $ is not finite, so in both cases $ P_{Im F} \mathcal{A}  P_{Im F} $ is a properly infinite von Neumann algebra. By applying \cite[Corollary 25]{BJMA} and Lemma  \ref{r12 l16} together with the proof of Lemma \ref{r12 l15} in the special case when $ F = D ,$ we deduce the result.
\end{proof}
\begin{definition}
	Let $ F \in \mathcal{A} .$ Then $ F $ is said to be an upper semi-$\mathcal{A} -$ B- Fredholm operator if there exists some $ n \in \mathbb{N} $ such that $ Im F^{m} $ is closed for all $ m \geq n , P_{Im F^{n} }$ is not finite and $ F P_{Im F^{n}} \in \mathcal{K} \Phi_{+} (P_{Im F^{n}} \mathcal{A}  P_{Im F^{n}} ) .$ Similarly, $ F $ is said to be a lower semi-$\mathcal{A} -$ B- Fredholm operator if the above conditions hold except that in the this case we assume that  $ F P_{Im F^{n}} \in \mathcal{K} \Phi_{-} (P_{Im F^{n}} \mathcal{A}  P_{Im F^{n}} ) .$ Finally, we say that $F$ is an $\mathcal{A} -$ B- Fredholm operator if the above conditions hold, however, in the this case we assume that  $ F P_{Im F^{n}} \in \mathcal{K} \Phi (P_{Im F^{n}} \mathcal{A}  P_{Im F^{n}} ) .$
\end{definition}

\begin{proposition} \label{bf1}
	If $F$ is an upper semi-$\mathcal{A} -$ B- Fredholm operator, $ n \in \mathbb{N} $ is the smallest $n$ such that $ Im F^{m} $ is closed for all $ m \geq n , P_{Im F^{n} }$ is not finite and $ F P_{Im F^{n}} \in \mathcal{K} \Phi_{+} (P_{Im F^{n}} \mathcal{A}  P_{Im F^{n}} ) ,$  then $ P_{Im F^{m} }$ is not finite and $ F P_{Im F^{m}} \in \mathcal{K} \Phi_{+} (P_{Im F^{m}} \mathcal{A}  P_{Im F^{m}} ) ,$  for all $ m \geq n .$ The analogue statement holds for lower semi-$\mathcal{A} -$ B- Fredholm operators. Finally, if $F$ is $\mathcal{A} -$ B- Fredholm operator, $ n \in \mathbb{N} $ is the smallest $n$ such that $ Im F^{m} $ is closed for all $ m \geq n , P_{Im F^{n} }$ is not finite and $ F P_{Im F^{n}} \in \mathcal{K} \Phi (P_{Im F^{n}} \mathcal{A}  P_{Im F^{n}} ) ,$  then $ P_{Im F^{m} }$ is not finite, $ F P_{Im F^{m}} \in \mathcal{K} \Phi (P_{Im F^{m}} \mathcal{A}  P_{Im F^{m}} ) ,$  for all $ m \geq n $ and $ index \text{ } FP_{Im F^{m}} = index \text{ } FP_{Im F^{n}} $ for all $ m \geq n .$
	\end{proposition} 
\begin{proof}
	By applying Lemma \ref{r12 l17} on the operator $ F P_{ Im F^{n}} $ and proceeding inductively, we deduce the statements in the proposition. 
\end{proof}

For an $\mathcal{A}$-$B$-Fredholm operator $F$ we set $\text {\rm index } F= \text {\rm index } { FP_{  Im F^{n }  }} ,$ where $n$ is the smallest $n$ such that $Im F^{m}$ is closed for all $m \geq n$ and such that $ FP_{  Im F^{n }  } \in  \mathcal{K} \Phi (P_{Im F^{n}} \mathcal{A}  P_{Im F^{n}} ) .$ \\

The next theorem is an analogue of \cite[Theorem 8]{IS5} in the setting of $\mathcal{A}$-Fredholm operators. 
\begin{theorem}\label{finitepeturb}
	 Let $T$ be an $\mathcal{A}$-$B$-Fredholm operator  and suppose that $m \in \mathbb{N}$ is the smallest $m$ such that $TP_{ Im T^{m}}  $ is an $\mathcal{A}$-Fredholm operator, $Im T^{n}  $ is closed for all $ n \geq m $ and $P_{ Im T^{m}} $ is not finite. Let $F$ be a finite operator and suppose that $ Im (T+F)^{n}  $ is closed for all $n \geq m   .$  Then $T+F$ is an $\mathcal{A}$-$B$-Fredholm operator and $\text {\rm index }  (T+F)=\text {\rm index }  T.$
\end{theorem}

\begin{proof}
	Let $ \tilde F = (T+F)^{m} - T^{m}.$ Then $P_{\ker \tilde F} \in \mathcal{A} ,$ which gives that $ P_{ \overline{ Im T^{m} P_{\ker \tilde F}} } \in \mathcal{A} .$ Now, by Lemma \ref{r12 l13} $ P_{ \overline{ Im \tilde F} } \in Proj_{0} (\mathcal{A}) ,$ hence $ P_{\ker \tilde F ^ { \perp }} \in Proj_{0} (\mathcal{A}) $ since $ P_{\ker \tilde F^{ \perp }} \sim  P_{ \overline{ Im \tilde F} } .$ As in the proof of \cite[Theorem 8]{IS5}  we can write  $ Im T^{m} $ and $ Im (T+F)^{m} $ as $$ Im T^{m} = \overline{ Im T^{m} P_{\ker \tilde F} } \oplus N, Im (T+F)^{m} = \overline{ Im T^{m} P_{\ker \tilde F} } \oplus N^{\prime} , $$ for some closed subspaces $ N, N^{\prime} ,$  so $ P_{N} , P_{N^{\prime}} \in \mathcal{A} $ since they are difference of projections that belong to $ \mathcal{A} .$ Then, again by the proof of \cite[Theorem 8]{IS5}, we obtain that 
	 $$ N = Im P_{ \overline{ Im T^{m} P_{\ker \tilde F}} ^{ \perp }} T^{m} P_{\ker \tilde F^{ \perp }}, N^{\prime} = Im P_{ \overline{ Im T^{m} P_{\ker \tilde F}} ^{ \perp }} (T+F )^{m} P_{\ker \tilde F^{ \perp }} .$$
	  Since $ P_{\ker \tilde F^{ \perp }} \in Proj_{0} (\mathcal{A}) ,$ we must have that $ P_{N} , P_{N^{\prime}} \in Proj_{0} (\mathcal{A}) $ by Lemma \ref{r12 l13}. Observe that since $P_{ Im T^{m} } $ is not finite and $ P_{N} $ is finite, we must have that $ P_{ \overline{ Im T^{m} P_{\ker \tilde F}} } $ is properly infinite. Hence, $ P_{Im (T+F)^{m}} $  is properly infinite because $ P_{ \overline{ Im T^{m} P_{\ker \tilde F}} } \leq  P_{Im (T+F)^{m}} . $ Since $ T^{m} P_{ Im T^{m} }\in \mathcal{K}\Phi (  P_{ Im T^{m} } \mathcal{A}  P_{ Im T^{m} } ) ,$  by \cite[Corollary 8]{BJMA} we have that $$ P_{ \overline{ Im T^{m} P_{\ker \tilde F}} } T P_{ \overline{ Im T^{m} P_{\ker \tilde F}} } \in \mathcal{K}\Phi ( P_{ \overline{ Im T^{m} P_{\ker \tilde F}} } \mathcal{A} P_{ \overline{ Im T^{m} P_{\ker \tilde F}} } ) $$ since $ P_{ \overline{ Im T^{m} P_{\ker \tilde F}} } = P_{ Im T^{m} } - P_{N} $ and $ P_{N} \in Proj_{0} (\mathcal{A}) .$ In addition, $$ \text {\rm index } T  P_{ Im T^{m} } = \text {\rm index } P_{ \overline{ Im T^{m} P_{\ker \tilde F}} } T P_{ \overline{ Im T^{m} P_{\ker \tilde F}} } .$$ Next, $ P_{ \overline{ Im T^{m} P_{\ker \tilde F}} } F P_{ \overline{ Im T^{m} P_{\ker \tilde F}} } $ is a finite operator, hence $$ P_{ \overline{ Im T^{m} P_{\ker \tilde F}} } (T+F ) P_{ \overline{ Im T^{m} P_{\ker \tilde F}} } \in \mathcal{K}\Phi ( P_{ \overline{ Im T^{m} P_{\ker \tilde F}} } \mathcal{A} P_{ \overline{ Im T^{m} P_{\ker \tilde F}} } ) $$ and $$ \text {\rm index } P_{ \overline{ Im T^{m} P_{\ker \tilde F}} } T P_{ \overline{ Im T^{m} P_{\ker \tilde F}} } = \text {\rm index } P_{ \overline{ Im T^{m} P_{\ker \tilde F}} } (T + F ) P_{ \overline{ Im T^{m} P_{\ker \tilde F}} } .$$ Now,  since $ Im (T+F)^{m} = \overline{ Im T^{m} P_{\ker \tilde F} } \oplus N^{\prime} ,$ we have that  $ P_{Im (T+F)^{m}} P_{ \overline{ Im T^{m} P_{\ker \tilde F}} } = P_{ \overline{ Im T^{m} P_{\ker \tilde F}} } .$ Therefore, $$P_{ \overline{ Im T^{m} P_{\ker \tilde F}} } (T + F ) P_{ \overline{ Im T^{m} P_{\ker \tilde F}} } = P_{ \overline{ Im T^{m} P_{\ker \tilde F}} } (T + F ) P_{Im (T+F)^{m}} P_{ \overline{ Im T^{m} P_{\ker \tilde F}} } =$$ $$ P_{ \overline{ Im T^{m} P_{\ker \tilde F}} } P_{Im (T+F)^{m}} (T + F ) P_{Im (T+F)^{m}} P_{ \overline{ Im T^{m} P_{\ker \tilde F}} } .$$ Moreover, since $ P_{N^{\prime}} \in Proj_{0} (\mathcal{A}) ,$ by \cite[Corollary 8]{BJMA} we have that  $$ (T + F ) P_{Im (T+F)^{m}} \in \mathcal{K}\Phi (  P_{ Im (T+F)^{m} } \mathcal{A}  P_{ Im (T+F)^{m} } ) .$$ In addition,  $$ \text {\rm index } (T + F ) P_{Im (T+F)^{m}} = \text {\rm index } P_{ \overline{ Im T^{m} P_{\ker \tilde F}} } T P_{ \overline{ Im T^{m} P_{\ker \tilde F}} } = \text {\rm index } T  P_{ Im T^{m} } .$$
	\end{proof}

\begin{remark}
	We notice that, in fact, by Proposition \ref{bf1}, it suffices to only assume in Theorem \ref{finitepeturb} that there exists some $N \geq m$ such that  $ Im (T+F)^{n}  $ is closed for all $n \geq N .$
\end{remark}	
Finally, the next proposition can be proved in a similar way as \cite[Proposition 4.3]{IS13}.

\begin{proposition}          %    \textbf{\underline{DP P26}}\\
	Let  $F,D \in \mathcal{A} $ satisfying that $FD=DF.$ Suppose that there exists an $n \in \mathbb{N} $ such that $Im (DF)^{m}$ is closed for all $m \geq n $ and in addition for each $m \geq n  $ we have that $Im F^{m+1}D^{m} $ and $Im D^{m+1}F^{m} $ are closed. If $F$ and $D$ are upper (lower) semi-$\mathcal{A}$-$B$-Fredholm, then $DF$ is upper (lower) semi-$\mathcal{A}$-$B$-Fredholm. If $F$ and $D$ are $\mathcal{A}$-$B$-Fredholm, then $DF$ is $\mathcal{A}$-$B$-Fredholm and $\text {\rm index }  DF=\text {\rm index }  D+\text {\rm index }  F.$ 
\end{proposition}

	\section{Disjoint Furstenberg transitivity in Banach bimodules}
The main purpose of this section is to extend the results from \cite{AOFA, arxiv} to some new classes of Banach bimodules that have not been considered in \cite{AOFA, arxiv}. 
The approach in this section is strongly motivated by the approach in \cite{AOFA, arxiv}, however, it provides also certain modifications of the approach in \cite{AOFA, arxiv}. The modified approach which we be presented in this section is  adopted for some new classes of Banach bimodules. 

We recall first the following definition.
\begin{definition}\label{glavna-definicija} \cite[Definition 4.1]{filomat}
	Let $\mathcal{S}$ be a set, $\mathcal{F}$ be a family of subsets of $\mathcal{S} $  such that $ \varnothing \notin \mathcal{F} $  
	and let $$\{ T_{t,1} \}_{t \in \mathcal{S}}, \dots,
	\{ T_{t,N} \}_{t \in \mathcal{S}}$$ be families of bounded linear operators on a Banach space $X$.  
	We say that $\{ T_{t,1} \}_{t \in \mathcal{S}}, \dots, \{ T_{t,N} \}_{t \in \mathcal{S}}$ are disjoint $\mathcal{F}$-semi-transitive, or shortly $d\mathcal{F}$-semi-transitive, if  
	for every collection of non-empty open subsets $\mathcal{O}, V_{1}, \dots, V_{N}$ of $X$,  
	there exists some $ F \in \mathcal{F} $  such that for all  $ t \in F$ there exists some $  \lambda_{t} \in \mathbb{R}^+  $  satisfying that 
	$$
	\mathcal{O} \cap \lambda_{t} T_{t,1}^{-1} (V_{1}) \cap \dots \cap \lambda_{t} T_{t,N}^{-1}(V_{N}) \neq \varnothing.
	$$
	
\end{definition}

For remarks on the relations between disjoint supercyclicity and disjoint disjoint $\mathcal{F}$-semi-transitivity, we refer to \cite{AOFA}.

Let now $\mathcal{M}$ be a Banach bimodule over a unital normed algebra $\mathcal{A}_1$.
We will say in this case that $\mathcal{M}$ satisfies \textit{ the condition (E) } with respect to $\mathcal{A}_1$ if there exists a constant $ K> 0$ such that  $$ \| ba \| \leq K \| b \|_1 \| a \|, $$for all $ a \in \mathcal{M}$ and $ b \in \mathcal{A}_1 $ where $ \| \cdot \|_1 $ denotes the norm on $\mathcal{A}_1$.

\text{ }

Also, we will assume that there exists a set  $\{p_\alpha\}_\alpha$  in $ \mathcal{A}_1 $ satisfying that given any open subset $O$ of $\mathcal{M}$ there exists some $ x \in \mathcal{M} $ and some  $ p_{\alpha_0} \in \{p_\alpha\}_\alpha$ such that  $ xp_{\alpha_0}^{2} \in  O.$ This condition will be called the \textit{condition (P*)}.

\text{ }

We recall also the following conditions from \cite{AOFA, arxiv}:

For a set $\mathcal{S},$ let $\mathcal{F}$ be a family of subsets of $\mathcal{S}$ such that $ \varnothing \notin \mathcal{F} $, and let
$$
\{\Phi_{t,1}\}_{t\in \mathcal{S}},\ \{\Phi_{t,2}\}_{t\in \mathcal{S}},\ \ldots,\ \{\Phi_{t,N}\}_{t\in \mathcal{S}}
$$
be families of isometric algebra isomorphisms of $\mathcal{A}_1,$ and $$
\{\Psi_{t,1}\}_{t\in \mathcal{S}},\ \{\Psi_{t,2}\}_{t\in \mathcal{S}},\ \ldots,\ \{\Psi_{t,N}\}_{t\in \mathcal{S}}
$$ be families of bounded linear isomorphisms of $ \mathcal{A} $ such that for all $ x, y \in \mathcal{A}_1 $ and $ a \in \mathcal{M} $ it holds that $ \Psi_{t,j}(xay) = \Phi_{t,j}(x)\Psi_{t,j}(a) \Phi_{t,j}(y) $   for all $t\in \mathcal{S} $ and $j \in \lbrace 1,\dots , N \rbrace.$ We will assume that the system
$$
\{\Phi_{t,1}\}_{t\in \mathcal{S}},\ \{\Phi_{t,2}\}_{t\in \mathcal{S}},\ \ldots,\ \{\Phi_{t,N}\}_{t\in \mathcal{S}}
$$
is disjoint aperiodic with respect to $\{p_\alpha\}_\alpha$, that is for each fixed $p_{\alpha}$ and every $H\in \mathcal{F}$ there exist some $F \subseteq H$ with $F \in \mathcal{F}$ such that
$$
p_{\alpha} \, \Phi_{t,l}(p_{\alpha}) = 0, 
\quad 
p_{\alpha} \, \Phi_{t,l}^{-1}(p_{\alpha}) = 0, 
\quad \text{and} \quad
p_{\alpha} \, \Phi_{t,l}\!\big( \Phi_{t,r}^{-1}(p_{\alpha}) \big) = 0
\quad \text{for all }  t \in F
$$
$ \text{ and all } r, l \in \{1, \dots, N\} \text{ with } r \neq l.
$
Further, we will assume that for each $\alpha$ and all $t \in \mathcal{S}$, 
$r, l \in \{1, \dots, N\}$, and $a \in \mathcal{M}$, 
it holds that
$$
\| a \, p_{\alpha} \| \le \| a \|, 
\quad 
\| a \, \Phi_{t,l}(p_{\alpha}) \| \le \| a \|,
$$
$$
\| a \, \Phi_{t,l}^{-1}(p_{\alpha}) \| \le \| a \|,
\quad 
\quad 
\| a \, \Phi_{t,l}\big(\Phi_{t,r}^{-1}(p_{\alpha})\big) \| \le \| a \|. 
$$
This condition will be called \textit{ the condition (R) } throughout this section.

For each $r,l \in \lbrace 1, \dots ,N \rbrace$ we let $\lbrace \parallel \cdot \parallel_{(t,l)} \rbrace_{t \in \mathcal{S}} $ and $\lbrace \parallel \cdot \parallel_{(t,l,-1)} \rbrace_{t \in \mathcal{S}} $ of be the families of norms 
on $\mathcal{M}$ such that  
$$ \parallel \Psi_{t,l} (a) \parallel_{(t,l)} = \|a\| $$ 
and  $$ \parallel \Psi_{t,l}^{-1} (a) \parallel_{(t,l,-1)} = \|a\| $$
for each $  a \in \mathcal{M}.$ Also, for each pair of distinct $r,l \in \lbrace 1 , \dots , N \rbrace$ we let $$\lbrace \parallel \cdot \parallel_{(t,r,l)} \rbrace_{t \in \mathcal{S}} $$ be families
of norms on $\mathcal{M}$ such that
\[
\left\|
\Psi_{t,l}
\left(
\Psi_{t,r}^{-1}(a)
\right)
\right\|_{\lbrace t,r,l\rbrace}
=
\|a\|
\]
for all $a\in\mathcal{M}$.

We now set up families 
$\{ b_{t,1} \}_{t \in \mathcal{S}}, \cdots,\{ b_{t,N} \}_{t \in \mathcal{S}}$
of elements in $\mathcal{A}_1$, 
and for each $t \in \mathcal{S}$ and $l \in \{1, \dots, N\}$, 
we define the operator
$$
T_{t,l} : \mathcal{M} \to \mathcal{M}
\quad \text{by} \quad
T_{t,l}(a) = b_{t,l} \, \Phi_{t,l}(a).
$$

Due to the condition $(E)$, 
it is not hard to check that $T_{t,l}$ 
is a bounded linear operator on $\mathcal{M}$.  
Further, we will assume that there exist a greater algebra $ \tilde{\mathcal{A} } $ such that $\mathcal{A}_1 \subset \tilde{\mathcal{A} }$ and such that for each index $ \alpha $ and every $  t \in \mathcal{S}, l \in \lbrace 1,\dots , N \rbrace$ there exists some $  b_{t,l, \alpha}^{-1} \in \tilde{\mathcal{A} } $  satisfying that $ b_{t,l, \alpha}^{-1} b_{t,l} a p_\alpha = b_{t,l}b_{t,l, \alpha}^{-1}a p_\alpha=a p_\alpha$ for every $ a \in \mathcal{M}.$ The algebra $\tilde{\mathcal{A} } $ does not need to be equipped with a norm, and $ \mathcal{M} $ does not need to be a Banach bimodule over $\tilde{\mathcal{A} } $ (but just over $ \mathcal{A}_1 $). Also, we will assume that $  b_{t,l, \alpha}^{-1} b_{t,l} \in \mathcal{A}_1 $ for each index $ \alpha $ and every $  t \in \mathcal{S}, l \in \lbrace 1,\dots , N \rbrace.$ 

Finally, we will assume that for each $ t \in \mathcal{S} $ and $l \in \{1, \dots, N\}$ it holds that   $ b_{t,l} a = 0 $ if and only if $ a = 0.$ Such condition will be called \textit{the condition (C)} throughout the manuscript.

By this new approach, we are now able to provide the following characterization.
\begin{theorem}\label{algebra-mere}
	Under the above notation and assumptions, the following statements are equivalent.
	\begin{enumerate}
		\item The families $\{ T_{t,1} \}_{t \in \mathcal{S}}, \cdots,\{ T_{t,N} \}_{t \in \mathcal{S}}$ are $d \mathcal{F}$-semi-transitive.
		
		\item For each fixed $ p_\alpha ,$ every fixed $ u,v_1, \cdots, v_N \in \mathcal{M} $  and each fixed $ \varepsilon > 0 $ there exists some 
		$ F \in  \mathcal{F} $ and families $\{ d_t \}_{t \in F}$, $ \{ g_{t,1} \}_{t \in F} \dots , \{ g_{t,N} \}_{t \in F} $ in $ \mathcal{M} $
		such that $ g_{t,l} \in b_{t,l}\mathcal{M}p_\alpha $ for all $ t \in F, l \in \lbrace 1,\dots , N \rbrace,$ and
		$$
		\|d_t - u p_\alpha^2   \| < \varepsilon, \quad    \|g_{t,l} - v_l p_\alpha^2   \| < \varepsilon,
		$$
		$$
		\|b_{t,r, \alpha}^{-1} g_{t,r}   \|_{(t,r)} \bigl   \|\Phi_{t,l}^{-1}(b_{t,l})\,d_t\bigr   \|_{(t,l, -1)} < \varepsilon^{2}
		\quad \text{for all } t \in F \text{ and } r,l \in \{ 1, \ldots, N \}.
		$$
		Moreover, for each distinct $ r,l \in \{ 1, \ldots, N \}$ and $t \in F$ it holds that
		$$
		\parallel 
		\Phi_{t,l}\Bigl(\Phi_{t,r}^{-1}\, ( b_{t,r}\,) \Bigr)   b_{t,l, \alpha}^{-1}g_{t,l} \, \,\parallel_{(t,,r,l)} < \varepsilon.
		$$
	\end{enumerate}
	
\end{theorem}

\begin{proof}
	The proof follows the lines of the proof of \cite[Theorem 3.1]{AOFA, arxiv} with certain modifications. For the convenience, we give the full proof.
	
	We prove first  $ (1) \Rightarrow (2)$. Since the families $\{ T_{ \Phi_{t,1},b_{t,1} } \}_{t \in S}  , \dots ,   \{ T_{ \Phi_{t,N},b_{t,N} } \}_{t \in S}$ 
	are $d \mathcal{F}$-semi-transitive, given $\varepsilon > 0,$ some fixed $ p_{\alpha} $ and  $ u,v_1, \cdots, v_N \in \mathcal{M} $  we can find 
	$a_t \in \mathcal{M}$ and $\lambda_{t}  \in \mathbb{R}^{+}$  such that $   \| a_t - u p_\alpha   \| < \varepsilon$ and 
	$$
	\|\lambda_{t} T_{ \Phi_{t,l}},b_{t,l}  (a_t)  -v_l  p_\alpha     \| < \varepsilon   \text{ for all } t \in H \text{ and some } H \in \mathcal{F}.
	$$

	Moreover, since
	$
	\bigl\{\ {\Phi_{t,l}} \bigr\} _{  t \in S, \; 1\le l \le N }          
	$
	is a disjoint aperiodic system, we can find some $ F \in \mathcal{F}$ with $ F \subseteq H $ such that 
	$
	p_{\alpha}  \!   \Phi_{t,l}^{-1}( p_{\alpha} )  =0, \;  
	p_{\alpha}  \!   \Phi_{t,l}( p_{\alpha} )  =0, \;  
	$
	\noindent and
	$
	p_\alpha	\Phi_{t,r}( \Phi_{t,l}^{-1}( p_{\alpha} ) ) =0, \; \text{for all } \; t\in F \; \text{ and } r,l \in \{ 1, \dots N\} \text{ with } r \neq l.
	$
	Hence, we get 
	
	$$
	\qquad
	\frac{1}{|\lambda_t|}\,
	\left   \|
	\Phi_{t,r}^{-1}\!\bigl(b_{t,r, \alpha}^{-1} b_{t,r}\bigr)\, a_{t}\, \Phi_{t,r}^{-1}(p_{\alpha})\, \lambda_{t}
	\right   \|
	=\frac{1}{|\lambda_t|}\, \left   \|
	\! b_{t,r, \alpha}^{-1}\,
	\!\bigl(b_{t,r}\, \Psi_{t,r} (a_{t})\, (p_{\alpha})\, \bigr) \lambda_{t} 
	\right   \|_{(t,r)}
	$$
	$$
	= \left   \|
	a_t\,\Phi_{t,r}^{-1}(p_{\alpha})
	\right   \| =\left   \|(a_{t}-up_{\alpha})\,\Phi_{t,r}^{-1}(p_{\alpha})
	\right   \|
	\le 
	\left   \|a_{t}-up_{\alpha}\right   \|
	< \varepsilon .
	$$ and 
	$$
	\left   \|
	\lambda_t\, b_{t,r}\,\Psi_{t,r}(a_{t})\,p_{\alpha}
	\;-\; v_r p_{\alpha}^{2}
	\right   \|
	\le
	\left   \|
	\lambda_t\, b_{t,r}\,\Psi_{t,r}(a_t) - v_r p_{\alpha}
	\right   \|
	=
	\left   \|
	\lambda_{t}\,T_{t,r}(a_{t}) - v_r p_{\alpha}
	\right   \|
	< \epsilon,
	\qquad 
	$$ for all  $t\in F$ and $ r\in\{1,\dots,N\}$. \\
	For each $t\in F$ and $r\in\{1,\dots,N\},$  we put $$	g_{t,r} = \lambda_t\, T_{t,r}(a_t) p_\alpha= \lambda_t b_{t,r} \Psi_{t,r}(a_t) p_\alpha .$$ 
	Then $ g_{t,r} \in b_{t,r}\mathcal{M} p_{\alpha}$ for all $ t \in F, r \in \lbrace 1,\dots , N \rbrace,$. Moreover,$$
	\|g_{t,r}-v_r p_{\alpha}^{2}   \| \;<\; \varepsilon$$ and
	$$
	\frac{1}{|\lambda_t|}\,   \|b_{t,r, \alpha}^{-1} g_{t,r}   \|_{(t,r)}
	\;<\; \varepsilon .$$
	
	Similarly, for each $ t\in F $ and	$ l\in\{1,\dots,N\},$ we have 
	$$
	\|\lambda_t\, b_{t,l}\,\Psi_{t,l}(a_t)\,\Phi_{t,l}(p_\alpha)   \|
	\;=\;    \|(\lambda_t\, b_{t,l}\,\Psi_{t,l}(a_t)-v_l p_\alpha)\,\Phi_{t,l}(p_\alpha)   \|
	\;\le\;    \|\lambda_t\, T_{t,l}(a_t)-v_l p_\alpha   \|
	\;<\; \varepsilon ,$$
	
	and,
	$$ \| a_t p_\alpha - u p_{\alpha}^{2}   \| \leq\| a_t  - u p_{\alpha}   \| \;<\; \varepsilon .$$
	For each $ t\in F$ and $ l\in\{1,\dots,N\},$
	set $ d_t = a_t p_\alpha. $
	Then   $$  \|d_t-u p_\alpha^{2}   \| \;<\; \varepsilon $$
	and 
	$$
	|\lambda_t|\,\bigl   \|\Phi_{t,l}^{-1}(b_{t,l})\,d_t\bigr   \|_{(t,l, -1)}
	\;=\;    \|\lambda_t\, b_{t,l}\,\Psi_{t,l}(a_t p_\alpha)   \|
	\;<\; \varepsilon .
	$$
	% --- end snippet ---

	Hence, for all $t\in F$ and $ r,l\in\{1,\dots,N\},$  we have
	$$\| b_{t,r, \alpha}^{-1} g_{t,r}    \|_{(t,r)}     \| \Phi_{t,l}^{-1} (b_{t,l}) d_{t}   \|_{(t,l, -1)} < \varepsilon^{2}. $$

	Further, for all $t\in F$ and each distinct $r,l\in\{1,\dots,N\}$  we get
	$$
	\Bigl\|\, \Phi_{t,l} \Bigl(\Phi_{t,r}^{-1}\!\bigl(b_{t,r}\bigr)\Bigr)\, b_{t,l, \alpha}^{-1} g_{t,l} \Bigr\|_{(t,r,l)}
	\;=\; \Bigl\|\,\!b_{t,r}\, \Psi_{t,r} \Psi_{t,l}^{-1}\!\bigl(b_{t,l, \alpha}^{-1}
	g_{t,l}\bigr)\Bigr\|
	$$
	$$
	=\bigl\|\,\lambda_t\, b_{t,r}\,\Psi_{t,r}(a_t)\,
	\Phi_{t,r}\!\bigl(\Phi_{t,l}^{-1}(p_\alpha)\bigr)\bigr\| 
	$$
	$$
	=\; \bigl\|\,(\lambda_t\, T_{t,r}(a_t)-v_r p_\alpha)\,
	\Phi_{t,r}\!\bigl(\Phi_{t,l}^{-1}(p_\alpha)\bigr)\bigr\|
	\;<\; \varepsilon ,
	$$
	
	\noindent
	where we have used that
	$
	p_\alpha\,\Phi_{t,r}\!\bigl(\Phi_{t,l}^{-1}(p_\alpha)\bigr)=0
	\quad\text{for all } t\in F
	\text{ and each distinct } r,l\in\{1,\dots,N\}.
	$

	Now we prove the implication $(2)\Rightarrow(1)$. 
	Let $O,V_1,\dots,V_N$ be non-empty open subsets of $\mathcal{M}$. 
	Then, by the condition (P*) we can find a non-zero $u \in O$, non-zero $v_l\in V_l$ for each 
	$l\in\{1,\dots,N\}$ and some  $\alpha$ such that
	$ u p_\alpha^{2} \in O $
	and
	$ v_l p_\alpha^{2}\in V_l\setminus\{0\} $
	for each $l=1,\dots,N$. Clearly, this gives that $ u p_\alpha^2 \neq 0 $ and $ v_l p_\alpha^2 \neq 0 $ for each $l=1,\dots,N$. Also, since $O\setminus\{0\},\,V_1\setminus\{0\},\dots, V_N\setminus\{0\}$ are open and non-empty,
	there is some $\delta>0$ such that the $\delta$-neighbourhood of $u p_\alpha^{2} $ is
	contained in $O \setminus\{0\}$, and the $\delta$-neighbourhood of $v_l p_\alpha^{2}$ is contained
	in $V_l\setminus\{0\}$ for each $l=1,\dots,N$.

	$$ \varepsilon
	:=\min\!\left\{ 
	\frac{\delta}{4 N}, \| u p_\alpha^2  \|, \| v_1 p_\alpha^2  \|, \cdots , \| v_N p_\alpha^2  \|
	\right\}. $$
	
	Choose $F\in  \mathcal{F}$ and the families 
	$ \{ g_{t,1} \}_{t\in F}, \dots  , \{ g_{t,N} \}_{t\in F}, \{d_{t} \}_{t \in F} $ 
	satisfying the assumptions in $(2)$ with respect to $ \epsilon $ and $p_{\alpha} $.For each $ \alpha $ and every $ t \in F, l\in\{1,\dots,N\}  $ we define the map $ S_{t,l, \alpha}: b_{t,l}  \mathcal{M}p_\alpha \rightarrow \mathcal{M} $ by $$ S_{t,l, \alpha}(b_{t,l} a p_\alpha ) = \Psi_{t,l}^{-1}(b_{t,l, \alpha}^{-1} b_{t,l} a p_\alpha ) $$ for all $  a \in \mathcal{M} .$ If $ b_{t,l} a p_\alpha =0 $ for some $  a \in \mathcal{M}$, by the condition $(C) $ it follows that $ a p_\alpha =0 .$ However, $b_{t,l, \alpha}^{-1} b_{t,l} a p_\alpha = a p_\alpha $ (by the definition of $\alpha-$inverse), hence, since $ \Psi_{t,l} $ is a linear isomorphism, it follows that $ S_{t,l, \alpha}(b_{t,l} p_\alpha a)=0, $ so $S_{t,l, \alpha} $ is well-defined.
	
	Next we observe that for each $t\in F$ and $l\in\{1,\dots,N\}$ it holds that
	$$
	\parallel T_{t,l} (d_{t})\parallel
	= \parallel\, b_{t,l}\,\Psi_{t,l} \bigl(d_{t} \bigr)\parallel
	= \parallel\,  \Phi_{t,l}^{-1}(b_{t,l}) d_{t}\parallel_{(t,l, -1)}
	$$
	and
	$$
	\parallel S_{t,l, \alpha} \bigl(g_{t,l} \bigr)\parallel
	= \parallel\, \Psi_{t,l}^{-1}(b_{t,l, \alpha}^{-1} g_{t,l} ) \parallel
	= \parallel b_{t,l, \alpha}^{-1} g_{t,l} 
	\parallel_{(t,l)}.
	$$
	
	Moreover, for all $t\in F$ and each distinct $r,l \in \lbrace 1,\dots,N \rbrace$ we have
	$$
	\parallel T_{t,r}\!\bigl(S_{t,l, \alpha}(g_{t,l})\bigr)\parallel
	= \parallel\,  b_{t,r}\, \Psi_{t,r}\, ( \Psi_{t,l}^{-1}\, (b_{t,l, \alpha}^{-1} g_{t,l} \bigr) \bigr)\,\parallel
	$$
	$$
	= \parallel 
	\Phi_{t,l}\Bigl(\Phi_{t,r}^{-1}\, ( b_{t,r}\,) \Bigr)   b_{t,l, \alpha}^{-1}g_{t,l} \parallel_{(t,,r,l)}.
	$$
	Finally, since $ b_{t,l, \alpha}^{-1} $ is the $ \alpha$-inverse of $b_{t,l} ,$ it follows that $$ T_{t,l} \bigl(S_{t,l, \alpha}(g_{t,l})) = g_{t,l} $$ for all $t\in F$ and each $ l \in \lbrace 1,\dots,N \rbrace .$

	Since $\| d_t -u p_\alpha^2  \| <  \epsilon \leq  \| u p_\alpha^2  \|$, 
	it follows that $d_t  \neq 0.$ Similarly, $g_{t,l} \neq 0$  for all
	$t \in F$ and $l \in \{1, \dots, N\}$. 
	Hence, since by the assumption we have that for each $ t \in \mathcal{S} $ and $l \in \{1, \dots, N\}$ it holds that   $ b_{t,l} a = 0 $ if and only if $ a = 0,$ it follows that $T_{t, l}(d_t ) \neq 0 $ for each $ t \in F$ and $l \in \{1, \dots, N\}.$  Similarly, since $ b_{t,l, \alpha}^{-1} $ is the $ \alpha$-inverse of $b_{t,l} ,$  it follows that   $S_{t, l, \alpha}(g_{t,l}) \neq 0$ for all $t \in F$ and $l \in {1, \dots, N}.$

	For each $t\in F$, we put
	$$
	x_{t} := d_{t}  + \frac{\sqrt{\sum_{l=1}^{N}\,\parallel T_{t,l}(d_t )\parallel}}{\sqrt{\sum_{l=1}^{N}\,\parallel S_{t,l, \alpha}(g_{t,l})\parallel} \qquad } \, {\sum_{r=1}^{N}\, S_{t,r, \alpha}(g_{t,r} )}.
	$$
	
	%	\medskip
	By triangle inequality and choice of $ \varepsilon$ one can deduce that for each $t\in F$, we have
	$$
	x_{t} \in O, \ \text{and} \ 
	\frac{
		\sqrt{ \sum_{l = 1}^{N} \| S_{t, l, \alpha}( g_{t, l}  ) \| }
	}{
		\sqrt{ \sum_{l = 1}^{N} \| T_{t, l}( d_t  ) \| }
	}
	T_{t, r}(x_t) \in V_r.
	$$
	for all $r \in \lbrace 1, \dots , N \rbrace.$ 
\end{proof}

\bibliographystyle{amsplain}

\end{document}